\documentclass{cocv-trafique}

\newcommand\bna{\begin{eqnarray*}} 
\newcommand\ena{\end{eqnarray*}}

\newcommand\bnan{\begin{eqnarray}} 
\newcommand\enan{\end{eqnarray}}

\newcommand\bnp{\begin{proof}} 
\newcommand\enp{\end{proof}}

\newcommand\bneq{\begin{eqnarray*}\left\lbrace \begin{array}{rcl}}
\newcommand\eneq{\end{array} \right.\end{eqnarray*}}
\newcommand\bneqn{\begin{eqnarray}\left\lbrace \begin{array}{rcl}}
\newcommand\eneqn{\end{array} \right.\end{eqnarray}}

\newcommand\nor[2]{\left\|#1\right\|_{#2}}

\numberwithin{equation}{section}

\newtheorem{lemma}{Lemma}[section]
\newtheorem{theorem}[lemma]{Theorem}
\newtheorem{proposition}[lemma]{Proposition}
\newtheorem{corollary}[lemma]{Corollary}

\theoremstyle{definition}

\newtheorem{definition}[lemma]{Definition}
\newtheorem{remark}[lemma]{Remark}

\makeatletter
\def\keywords{
    \vspace{1ex}
    \noindent
    \if@twocolumn
      \small{\bf  Keywords}\/---$\!$    \else
      \begin{center}\small\ {\bf Keywords}\end{center}\quotation\small
    \fi}
\def\endkeywords{\vspace{0.6em}\par\if@twocolumn\else\endquotation\fi
    \normalsize\rm}
\makeatother

\newcommand{\G}{\ensuremath{\mathcal G}}

\renewcommand{\L}{\ensuremath{\mathcal L}}

\DeclareMathOperator{\phg}{phg}

\newcommand{\obs}{\ensuremath{b_\omega}}

\DeclareMathOperator{\loc}{loc}

\newcommand{\mb}[1]{\ensuremath{\mathbb{#1}}}
\newcommand{\N}{{\mb{N}}}

\newcommand{\R}{{\mb{R}}}
\newcommand{\C}{{\mb{C}}}

\newcommand{\eps}{\varepsilon}

\newcommand{\D}{\ensuremath{\mathscr D}}

\newcommand{\Bo}{\ensuremath{\mathcal{B}}}

\let \Re \relax
\DeclareMathOperator{\Re}{Re}
\let \Im \relax
\DeclareMathOperator{\Im}{Im}

\newcommand{\ovl}[1]{\overline{#1}}

\newcommand{\Con}{\ensuremath{\mathscr C}}
\newcommand{\Cinf}{\ensuremath{\Con^\infty}}
\newcommand{\Cinfc}{\ensuremath{\Con^\infty}_{c}}

\DeclareMathOperator{\supp}{supp}

\DeclareMathOperator{\dist}{dist}

\DeclareMathOperator{\Op}{Op}

\DeclareMathOperator{\id}{Id}

\newcommand{\transp}{\ensuremath{\phantom{}^{t}}}

\renewcommand{\d}{\ensuremath{\partial}}

\newcommand{\nhd}{neighborhood\xspace}

\newcommand{\wrt}{w.r.t.\@\xspace}

\newcommand{\eg}{e.g.\@\xspace}

\newcommand{\resp}{resp.\@\xspace}

\let \div \relax
\DeclareMathOperator{\div}{div}

\DeclareMathOperator{\length}{length}

\def\x{x}

\newcommand{\E}{\mathscr E}

\renewcommand{\H}{\ensuremath{\mathcal H}}

\usepackage{amsthm}
\usepackage{graphicx,color,subfig}
\usepackage{xspace}
\usepackage{mathrsfs}

\begin{document}
\title{Uniform observability estimates for linear waves}
\thanks{The first author is partially supported by the Agence Nationale de la Recherche under grants EMAQS ANR-2011-BS01-017-0 and IPROBLEMS ANR-13-JS01-0006.}
\thanks{The second author is partially supported by the Agence Nationale de la Recherche under grant GERASIC ANR-13-BS01-0007-01.}
%
\author{Camille Laurent}\address{CNRS UMR 7598 and UPMC Univ Paris 06, Laboratoire Jacques-Louis Lions, F-75005, Paris, France, email: laurent@ann.jussieu.fr}
\author{Matthieu L\'eautaud}\address{Universit\'e Paris Diderot, Institut de Math\'ematiques de Jussieu-Paris Rive Gauche, UMR 7586, B\^atiment Sophie Germain, 75205 Paris Cedex 13 France, email: leautaud@math.univ-paris-diderot.fr}

%
%
\begin{abstract}
 In this article, we give a completely constructive proof of the observability/controllability of the wave equation on a compact manifold under optimal geometric conditions. This contrasts with the original proof of Bardos-Lebeau-Rauch~\cite{BLR:92}, which contains two non-constructive arguments.
Our method is based on the Dehman-Lebeau~\cite{DL:09} Egorov approach to treat the high-frequencies, and the optimal unique continuation stability result of the authors~\cite{LL:15} for the low-frequencies.

As an application, we first give estimates of the blowup of the observability constant when the time tends to the limit geometric control time (for wave equations with possibly lower order terms).
Second, we provide (on manifolds with or without boundary) with an explicit dependence of the observability constant with respect to the addition of a bounded potential to the equation.
 \end{abstract}
\begin{resume} 
Dans cet article, nous proposons une d\'emonstration compl\`etement constructive de l'observa-bilit\'e/la contr\^olabilit\'e de l'\'equation des ondes sur une vari\'et\'e compacte, sous les conditions g\'eom\'etriques optimales.
Ceci contraste avec la preuve originelle de Bardos-Lebeau-Rauch~\cite{BLR:92}, qui contient deux arguments non constructifs.
Notre m\'ethode repose sur l'approche par Egorov de Dehman-Lebeau~\cite{DL:09} pour traiter les hautes fr\'equences, et sur le r\'esultat de stabilit\'e optimal pour le prolongement unique, obtenu par les auteurs dans~\cite{LL:15}, pour les basses fr\'equences.

Comme application, nous donnons tout d'abord des bornes sur l'explosion de la constante d'observa-bilit\'e lorsque le temps d'observation tends vers le temps minimal de contr\^ole g\'eom\'etrique (pour des \'equations d'ondes contenant \'eventuellement des termes d'ordre inf\'erieur). Enfin, nous estimons (sur des vari\'et\'es avec ou sans bord) la d\'ependance de la constante d'observabilit\'e par rapport \`a l'ajout d'un terme de potentiel born\'e dans l'\'equation.

\bigskip
\begin{center}
\textit{
In honor of Jean-Michel Coron on the occasion of his 30th birthday. 
}
\end{center}
 \end{resume}

\maketitle

\tableofcontents


\section{Introduction}

\subsection{Motivation}
This article is devoted to control and observation issues for the wave equation on a $n$-dimensional compact Riemannian manifold $(M,g)$ with or without boundary $\d M$. 
Denoting by $\Delta$ the nonpositive Laplace-Beltrami operator on $M$ and by $L$ the selfadjoint operator $- \Delta$ on $L^2(M)$ with Dirichlet boundary conditions, the general controllability problem in time $T>0$ is whether, for each data $ (u_0, u_1)$ one can find a control function $f$ such that the solution $u$ to
\begin{equation}
\label{e:KG-cont}
\begin{cases}
\d_t^2 u + L u = \obs  f\\
(u(0), \d_t u(0)) = (u_0, u_1)
\end{cases}
\end{equation}
satisfies $(u(T), \d_t u(T)) = (0, 0)$ (or, equivalently, any given state). In this equation the control $f$ acts on the state $u$ only in the set $\omega := \{\obs \neq 0\}$ where $\obs$ is, say, a continuous function. A classical functional analysis argument~\cite{DR:77} reduces this existence problem to that of finding (for the same time $T>0$) a constant $\mathfrak{C}_{obs}>0$, such that all solutions to 
\begin{equation}
\label{e:KG-obs}
\begin{cases}
\d_t^2 v + L^* v = 0 \\
(v(0), \d_t v(0)) = (v_0, v_1)
\end{cases}
\end{equation}
with $(v_0, v_1) \in H^1(M) \times L^2(M)$ satisfy the so called observability inequality
\bnan
\label{e:general-obs}
 \mathfrak{C}_{obs}  \int_0^T \|\obs v(t) \|_{H^1(M)}^2 dt \geq \E_1(v_0 , v_1) = \frac12 \left( \| v_0 \|_{H^1(M)}^2 + \| v_1 \|_{L^2(M)}^2  \right) .
\enan
Such an estimate translates that the full energy of the state $v$ (which is preserved through time) may be recovered from the sole observation on the (possibly small) set $\omega$ during the time interval $(0,T)$.

Not only the controllability problem and the observability problem are equivalent, but also, in case~\eqref{e:general-obs} holds, the constant $\mathfrak{C}_{obs}^\frac12$ bounds the norm of the control operator $(u_0, u_1) \mapsto f$ mapping to the data to be controlled the associated optimal control function $f$ (in appropriate spaces).

\bigskip
A first natural attempt at proving the energy inequality~\eqref{e:general-obs} in dimension $n\geq 2$ consists in multiplying~\eqref{e:KG-obs} by $\mathfrak{M} v$, where $\mathfrak{M}$ is an appropriate first order differential operator, and perform integrations by parts. Such ``multiplier methods'' have been developed in a large number of situations, leading to~\eqref{e:general-obs} under strong geometric conditions on $(\omega,T)$  (see the references~\cite{Lio:88,Kom:94}): basically, in the case where $M$ is an open subset of $\R^n$, given a point $x_0 \in \R^n$, it is required that $\omega$ contains a neighborhood of the points $x$ of the boundary where $(x-x_0) \cdot n(x) >0$, $n$ being the outgoing normal to $\d M$ and that $T> 2 \sup_{x \in M}|x-x_0|$ (the multiplier is $\mathfrak{M} = (x-x_0) \cdot \d_x$).

Another constructive proof uses global Carleman estimates (see e.g. \cite{DZZ:08,BDBE:13} , which amount to prove positivity properties for $P_\psi^* P_\psi$, where $P_\psi = e^{\tau \psi}(\d_t^2 + L^*)e^{-\tau \psi}$ is a conjugated operator. Here, $\tau$ is a large parameter and $\psi$ an appropriately chosen weight function. Unfortunately, global weights $\psi$ that give rise to positivity of $P_\psi^* P_\psi$ require that $(\omega,T)$  satisfy similar conditions as those coming from multiplier methods.

The advantage of these two direct computational methods is that the proofs are constructive and provide with effective bounds, that are uniform with respect to parameters.
However, though very effective, they present an important drawback: they require very strong and inappropriate geometric conditions
(see~\cite{Miller:02} for a geometric discussion on multiplier methods). Indeed, they do not capture the main features of wave propagation, stating roughly that most of the energy should travel along rays of geometric optics. 

\bigskip
The complete characterization of $(\omega,T)$ for which the observability inequality~\eqref{e:general-obs} holds was achieved in~\cite{BLR:88,BLR:92}: observability holds if and only if the Geometric Control Condition (GCC) does: every ray of geometric optics enters $\omega$ in the time interval $(0,T)$ (see also~\cite{BG:97} for the ``only if'' part). The proof of~\cite{BLR:88,BLR:92} is based on a compactness-uniqueness argument and splits into two parts:
\begin{enumerate}
\item Proving a relaxed observability inequality 
\bnan
\label{e:general-obs-relaxed}
 \mathfrak{C} \int_0^T \|\obs v(t) \|_{H^1(M)}^2 dt \geq \E_1(v_0 , v_1) - \mathfrak{C}'\E_0(v_0 ,v_1),
\enan
where $\E_0(v_0 ,v_1)= \frac12 \left( \| v_0 \|_{L^2(M)}^2 + \| v_1 \|_{H^{-1}(M)}^2  \right)$ is a weaker energy of the data (or, equivalently, of the state). This estimate translates the high-frequency behavior, and relies on the propagation of singularities for the wave equation (see Corollary \ref{cor:obsHFbis} and Lemma \ref{lm:obsHFbis} for a justification of the terminology ``high-frequency''). That $(\omega,T)$ should satisfy GCC is used in this step.
\item Getting rid of the additional term $\E_0(v_0 ,v_1)$ in~\eqref{e:general-obs-relaxed}. This step relies on unique continuation properties and requires less on $(\omega,T)$.
\end{enumerate}
\medskip
Both parts of the proof rely on the understanding of two fondamental properties of the wave equation: (1) propagation of high-frequencies along the rays of geometric optics, and (2) ``propagation'' of low-frequencies according to the tunnel effect.

In the original proof~\cite{BLR:88,BLR:92}, Step (1) relies on a closed graph theorem and the propagation of wave front sets. A variant of this proof, proposed by Lebeau~\cite{Leb:96}, relies instead on a contradiction argument and the propagation of microlocal defect measures of  G\'erard and Tartar~\cite{Gerard:91,Tartar:90}.
 
After reductio ad absurdum, Step (2) is then equivalent to proving that any solution of~\eqref{e:KG-obs} vanishing on the set $(0,T) \times \omega$ vanishes everywhere. This step can be performed using a global unique continuation results for waves: the global version of the Holmgren theorem, as proved by John~\cite{John:49} in the analytic setting, or by Tataru, Robbiano-Zuily and H\"ormander in~\cite{Tataru:95,RZ:98,Hor:97} in the general case. At the time~\cite{BLR:88,BLR:92} were written, this general unique continuation theorem for waves (under optimal geometric conditions) was not known in the non-analytic case. Bardos-Lebeau-Rauch managed to bypass this argument by using strongly estimate~\eqref{e:general-obs-relaxed} and studying the set of invisible solutions, which then reduces the problem to a classical unique continuation result for eigenfunctions of $L^*$.

\medskip
It is clear from this brief discussion that both steps are highly non constructive, so that the full proof may not seem well-suited for tracking the dependence/robustness of the observability constant $\mathfrak{C}_{obs}$ with respect to parameters (e.g. w.r.t. the observation time $T$, lower order terms added in the operator $L$...).

\medskip
The aim of this paper is to provide with a constructive proof under optimal geometric conditions. For this, we explain:
\begin{itemize}
\item how to replace Step (2) above by the optimal unique continuation estimates obtained by the authors in~\cite{LL:15};
\item on a compact manifold without boundary, how to replace step (1) using the analysis of Dehman and Lebeau~\cite{DL:09}.
\end{itemize}

We illustrate the interest of this approach by keeping track of some parameters in the analysis. Firstly, we give bounds on the blowup of the observability constant $\mathfrak{C}_{obs}$ as a function of the observation time $T$ when it goes to the limit control time associated to the open set $\omega$, namely $T \to T_{GCC}(\omega)^+$. Secondly, we provide with an explicit bound of the dependence of the observability constant $\mathfrak{C}_{obs}$ when adding to the equation a potential, i.e. taking $L = -\Delta  + c(x)$ in~\eqref{e:KG-cont}-\eqref{e:KG-obs}.

We also hope that the method we develop here might be used for other purposes (e.g. inverse problems, data assimilation, big data...) where getting uniform estimates might be of importance.

\subsection{Main results}
Before stating our results, let us recall some geometric definitions needed to formulate them (see also Appendix~\ref{app:geom}). For $E \subset M$, we define ``the largest distance from $E$ to a point in $M$'' by
\bnan
\label{e:def-L}
\mathcal{L}(M ,E):= \sup_{x_1 \in M}  \dist(x_1 ,  E)  .
\enan
We shall also use the notation
\bnan
\label{e:def-TUC}
T_{UC}(E) = 2 \mathcal{L}(M ,E) ,
\enan
which, in case $E$ is open, is the minimal time of unique continuation for the wave equation from the set $E$ (see~\cite{John:49} in the analytic setting
or~\cite{Tataru:95,RZ:98,Hor:97} in the general case). In turn, it also provides the optimal time of approximate controllability from the open set $E$.

Assume for a while that $\d M = \emptyset$.
According to~\cite{RT:74,BLR:92}, given an open set $\omega$ and a time $T>0$, we say $(\omega ,T)$ satisfies GCC if 
\bna
\text{for any } \rho \in S^*M, \text{ there exists } t \in(0,T) \textnormal{ so that }\pi(\varphi_t (\rho)) \in\omega ,
\ena
where,  $\varphi_t$ is the geodesic flow on $S^*M$ and $\pi$ the canonical projection $S^*M \to M$ (see Appendix~\ref{app:geom}). We also say that $\omega$ satisfies GCC if there is a time $T>0$ such that $(\omega,T)$ does. If $\omega$ satisfies GCC, then we may define the minimal control time associated to $\omega$ by
\bnan
\label{e:def-TGCC}
T_{GCC}(\omega) = \inf\{T>0 , (\omega,T) \text{ satisfies } GCC \} .
\enan
We also have 
\bna
T_{GCC}(\omega) &=& \sup \{\length(\Gamma), \Gamma \textnormal{ geodesic curve on $M$ with } \Gamma \cap \omega=\emptyset\}\\
&=&\inf \{T>0 \textnormal{ such that any geodesic curve $\Gamma$ with } \length(\Gamma)\geq T \text{ satisfies }\Gamma \cap \omega\neq \emptyset\}.
\ena
It can be proved that $T_{GCC}(\omega) \geq T_{UC}(\omega)$ (a proof is given in Appendix~\ref{app:geom} Lemma \ref{lmTUleqTGCC}).
 Given a continuous function $\obs$, we also define the constant
 \bnan
\label{e:def-KT}
\mathfrak{K}(T) = \min_{\rho \in S^*M} \int_0^T \obs^2 \circ \pi  \circ\varphi_t (\rho)dt  ,
\enan
which is the smallest average of the function $\obs^2$ along geodesics of length $T$. With this definition, we also have $T_{GCC}(\omega) = \inf\{T>0 , \mathfrak{K}(T) >0\} =  \sup\{T>0 , \mathfrak{K}(T) = 0\}$, with $\omega : = \{\obs \neq 0\}.$ 

In the case $\d M \neq \emptyset$, one may also define a (continuous) ``broken'' geodesic flow on the appropriate phase space (see~\cite{BLR:92}), and the above definitions still allow to express that $(\omega,T)$ satisfies GCC. When considering the boundary observation/control problem, we need the following definition~\cite{BLR:92}: given $\Gamma \subset \d M$ and $T>0$, we say that $(\Gamma,T)$ satisfy the Geometric Control Condition GCC$_\d$ if every generalized geodesic (i.e. ray of geometric optics) traveling at speed one in $M$ meets $\Gamma$  on a non-diffractive point in a time $t \in (0,T)$.

As already mentioned, we present here two different types of results: first estimate $\mathfrak{C}_{obs}$ as a function of time $T$ when $T\to T_{GCC}^+(\omega)$, and second to estimate $\mathfrak{C}_{obs}$ as a function of the potential $c(x)$, when taking $L = -\Delta  + c(x)$ in~\eqref{e:KG-cont}-\eqref{e:KG-obs}.

\bigskip
Our first results concern, in the case $\d M = \emptyset$, the behaviour of the constant $\mathfrak{C}_{obs}(T)$ as a function of the observation time $T$ when the latter is close to $T_{GCC}(\omega)$.
We first prove that the observability estimate~\eqref{e:general-obs} always fails for the critical time $T=T_{GCC}(\omega)$, and give an explicit blowup rate when $T \to T_{GCC}(\omega)^+$. In all what follows, we assume that $\obs \in \Cinf(M)$ (or, at least $\Con^k(M)$ for some large~$k$).
\begin{theorem}
\label{t:lower-obs}
Assume that $\d M = \emptyset$ and \eqref{e:general-obs} holds for all solutions of \eqref{e:KG-obs} with $L=- \Delta+1$. Then, $\mathfrak{K}(T)>0$ (i.e. $(\omega ,T)$ satisfies GCC) and we have $\mathfrak{C}_{obs}(T) \geq \mathfrak{K}(T)^{-1}$, where $\mathfrak{K}(T)$ is defined in~\eqref{e:def-KT}.
\end{theorem}
That is to say that the observability constant $\mathfrak{C}_{obs}(T)$ blows up at least like $\mathfrak{K}(T)^{-1}$ as $T \to T_{GCC}(\omega)^+$.

We also obtain an upper bound on this blowup rate. Namely, we shall prove the following uniform observability estimate.
\begin{theorem}[Uniform observation theorem]
\label{th:KG-obs} 
Assume that $\d M = \emptyset$, $\omega = \{ \obs \neq 0\}$ satisfies GCC and that $T_{UC}(\omega) < T_{GCC}(\omega)$. Then, for any $T_1 > T_{GCC}(\omega)$, there exist $C, \kappa >0$ such that for any $T \in  (T_{GCC}(\omega) , T_1]$, any $(v_0,v_1)\in H^1(M) \times L^2(M)$ and $v$ associated solution of~\eqref{e:KG-obs} with $L=- \Delta+1$, we have
\bna
 \E_1(v_0, v_1) \leq Ce^{\kappa \mathfrak{K}(T)^{-1} } \int_0^T \|\obs v (t)\|_{H^1(M)}^2 dt  ,
\ena
where $\mathfrak{K}(T)$ is defined in~\eqref{e:def-KT}.
\end{theorem}

Using the classical duality argument~\cite{DR:77,Lio:88,Cor:book}, we deduce the following uniform control result.
\begin{corollary}[Uniform control theorem]
\label{th:KG-cont} 
Under the assumptions of Theorem~\ref{th:KG-obs}, for any $T_1 > T_{GCC}(\omega)$, there exist $C, \kappa >0$ such that for any $T \in  (T_{GCC}(\omega) , T_1]$, any $(u_0 ,u_1)\in L^2(M) \times H^{-1}(M)$, there exists $f \in L^2(0,T; H^{-1}(M))$ with 
$$
\| f \|_{L^2(0,T; H^{-1}(M))}^2 \leq  Ce^{\kappa \mathfrak{K}(T)^{-1} } \E_0(u_0 ,u_1) ,
$$
 (where $\mathfrak{K}(T)$ is defined in~\eqref{e:def-KT}) such that the associated solution $u$ of~\eqref{e:KG-cont} satisfies $(u(T),  \d_t u(T)) = (0,0)$.
\end{corollary}

Note that from this result and a commutator argument (see~\cite{DL:09,EZ:10}) one may deduce a similar bound on the norm of the control in $L^2(0,T; H^{s-1}(M))$ for data in $H^s(M) \times H^{s-1}(M)$.

In dimension $n \geq 2$, the condition $T_{UC}(\omega) < T_{GCC}(\omega)$ is not very restrictive (and, in particular, is certainly generic with respect to the set $\omega$ or the metric $g$). Indeed, we prove in Section~\ref{s:compare-time} that $T_{UC}(\omega) = T_{GCC}(\omega)$ implies a very specific geometric situation. Roughly speaking, it shows that close to the points where the maximum of $\dist(x,M)$ is reached, $M\setminus \overline{\omega}$ is a closed geodesic balls of radius $T_{UC}(\omega)/2$. A precise statement is given in Lemma \ref{lmTUeqTGCC}.

\bigskip
The estimation of the cost of fast control has already been investigated in several situations: in finite dimension~\cite{Sei:88}, in different situations for the Schr\"odinger equation~\cite{Miller:04shrodinger,Miller:05,Lissy:15}, for the heat equation~\cite{Miller:04heat,Lissy:15}, for the Stokes equation~\cite{CSL:16}.

In all these cases, the equations under study are controllable in any time $T>0$ and the question is about to estimating how the observability constant blows up as $T \to 0^+$. We are not aware of any such results in the case of the wave equation, for which a minimal time is required to have observability.

Note that short time estimates of the control cost for the heat equation are also known to imply uniform estimates of the control for a transport-diffusion equation in the vanishing viscosity limit, see \cite{Lissy:12}. This problem was originally studied by Coron and Guerrero in~\cite{CG:05}. In this context, a minimal time also appears to obtain uniform observability. The question of getting uniform observability in the natural time related to the transport equation remains open.

\bigskip
The above results are particularly simple to write in the case of the wave equation $L=-\Delta$. However, they generalize (with some technicalities, but no additional conceptual difficulty) to wave equations with lower order terms.
In that context, we wish to consider the control problem~\eqref{e:KG-cont} in case the operator $L$ is a general time-dependent perturbation of $-\Delta$, defined by
\bnan
\label{e:def-op-elliptiq}
L u (t,x)= -\Delta u (t,x)+ b_0(t,x)\d_t u(t,x) + \langle d u (t,x), b_1(t,x) \rangle_x +c(t,x) u (t,x),
\enan
where $b_0$ and $c$ are smooth functions on $\R \times M$ and $b_1$ is a smooth time dependent vector field on $M$. Note that we may equivalently rewrite $\langle d u (x), b_1(x) \rangle_x = g_x( \nabla u (x), b_1(x) )$ (see Appendix~\ref{app:geom} for notations).

The adjoint observation problem is~\eqref{e:KG-obs} with
\bnan
\label{e:def-op-elliptiq-adj}
L^* v (t,x)= -\Delta v (t,x) - \ovl{b_0} (t,x) \d_t v (t,x) -  \langle d v (t,x), \ovl{b_1}(t,x) \rangle_x+ \left(\ovl{c}  -\d_t  \ovl{b_0}  - \div( \ovl{b_1})  \right) v (t,x) ,
\enan
and if $u$ is a smooth solution to~\eqref{e:KG-cont}-\eqref{e:def-op-elliptiq} and $v$ a smooth solution to~\eqref{e:KG-obs}-\eqref{e:def-op-elliptiq-adj}, we have the duality identity
\bna
\left[(\d_t u , v)_{L^2(M)} - ( u , \d_t v)_{L^2(M)} + (b_0 u , v)_{L^2(M)}\right]_0^T = \int_0^T (f, \obs v)_{L^2(M)} dt .
\ena

In this general setting, we obtain the same results as in the case $L=L^*=-\Delta$, with an analyticity assumption with respect to time on the coefficients, and where $\mathfrak{K}(T)$ has to be appropriately modified.

\begin{theorem}
\label{t:main-result-lot}
Assume that the coefficients of $b_0$, $c$ and $b_1$ are smooth and depend analytically  on the variable $t$.
Then, the analogues of Theorem~\ref{t:lower-obs} and Theorem~\ref{th:KG-obs} hold for Equation~\eqref{e:KG-obs}
with $L^*$ given by~\eqref{e:def-op-elliptiq-adj} and 
\bnan
\label{e:KTgeneral+}
\mathfrak{K}(T)  =  \min_{\rho \in S^*M} g_T^+(\rho) ,
\enan
where, denoting by $(x(s) , \xi(s)) = \varphi_s (x_0 ,\xi_0)$, we have
$$
g_T^+ (x_0 , \xi_0)=  \int_0^T  \obs^2(x(t))  \exp \left(\int_0^t \Re(b_0)(\tau, x(\tau)) + \left< \frac{\xi(\tau)}{|\xi(\tau)|_{x(\tau)}} , \Re(b_1)(\tau , x(\tau)) \right>_{x(\tau)}  d\tau \right) dt.
$$
\end{theorem}
In fact, the analogue of Theorem~\ref{t:lower-obs} (lower bound) does not require the analyticity in time of the coefficients. Analyticity in time (on the time interval $[0,T_1]$, where $T_1$ is given in the statement of Theorem~\ref{th:KG-obs}) is however strongly used in the proof of the analogue of Theorem~\ref{th:KG-obs} (upper bound) which relies on the unique continuation argument of~\cite{Tataru:95,RZ:98,Hor:97,LL:15}.
Note that this result would be the same if we replaced the observation equation $\d_t^2 v + L^* v = 0$ by $\d_t^2 v + L v = 0$. Indeed, the symbol $g_T^+ (x, \xi)$ (and thus the constant $\mathfrak{K}(T)$) only depends on $\Re(b_0)$ and $\Re(b_1)$.
Remark also that the damped wave equation corresponds to the case $c=0, b_1= 0$ and $b_0$ real valued.

\bigskip
Let us now consider the problem of obtaining a uniform observability constant $\mathfrak{C}_{obs}$ for perturbations of $- \Delta$ by a potential $c \in L^\infty(M)$. Here, we no longer assume that $M$ has no boundary, and our result work for boundary observation as well. 
In this setting, Dehman-Ervedoza~\cite{DE:14} proved that the constant $\mathfrak{C}_{obs}$ remains uniformly bounded for $\|c\|_{L^\infty}$ bounded. Here, we give an explicit bound. The purpose of the following results is to explicitly establish this bound.
We have a rough result for general potentials, and a refined one in case $c \in L^\infty_{\delta}(M)$, where
\bnan
L^\infty_{\delta}(M)= \left\{c \in L^\infty (M ; \R), \delta \|u\|_{L^2(M)}^2 \leq   \int_M |\nabla u|^2 + c|u|^2 , \text{ for all } u \in H^1_0(M) \right\} , \quad \text{for } \delta \geq 0 .
\enan
Remark that functions in $L^\infty_{\delta}(M)$ are real-valued. In case $\d M = \emptyset$, $H^1_0(M)$ stands for $H^1(M)$.

\begin{theorem}
\label{t:dep-pot}
Assumes that $(\omega, T)$ satisfies GCC, \resp that $(\Gamma, T)$ satisfies GCC$_{\d}$. Then, for any $c \in L^\infty(M)$, any $V_0 =  (v_0, v_1) \in H^1_0(M) \times L^2(M)$, and associated solution $v$ of
 \begin{equation}
\label{e:waves-b-eq-c}
\begin{cases}
\d_t^2 v -\Delta v + c(x)v = 0 , \\
v|_{\d M} =  0  , \quad \text{if } \d M \neq \emptyset \\
(v(0), \d_t v(0)) = (v_0, v_1) ,
\end{cases} 
\end{equation}
we have the estimates
\bnan
\label{e:waves-b-obs-c}
 \mathfrak{C}_{obs}  \int_0^T \| v(t) \|_{H^1(\omega)}^2 dt \geq \E_1(V_0) ,  
\enan
\resp  , 
\bnan
\label{e:waves-bb-obs-c}
 \mathfrak{C}_{obs}  \int_0^T \| \d_\nu v(t) \|_{L^2(\Gamma)}^2 dt \geq \E_1(V_0) . 
\enan
with $\mathfrak{C}_{obs} = \mathfrak{C}_{obs}(\|c\|_{L^\infty})$ where  $\mathfrak{C}_{obs}(r)= C\exp( \exp (C \sqrt{r}))$.

 If $c \in L^\infty_\delta(M)$, $\delta>0$, then $\mathfrak{C}_{obs} =  \mathfrak{C}_{obs}(\delta, \|c\|_{L^\infty})$ where $\mathfrak{C}_{obs}(\delta, r) = C\exp (C(1 + \delta^{-1/2})r)$ (where the constant $C>0$ does not depend on $\delta, r$).
\end{theorem}

Even the refined estimate in the case $c \in L^\infty_\delta(M)$ does not reached  the general conjecture of Duyckaerts-Zhang-Zuazua~\cite{DZZ:08}, being that $\mathfrak{C}_{obs}(r)$ should be of the form $C\exp (C r^\frac23)$ for all $c \in L^\infty (M)$ in dimension $n \geq 2$.
However, whereas the $C\exp (C r^\frac23)$ bound is proved in~\cite{DZZ:08} in case $(\omega,T)$ satisfy a mutliplier-type condition, to our knowledge, Theorem~\ref{t:dep-pot} is the first explicit bound under the sole GCC condition. We also refer to~\cite{Zua:93} for the dependence w.r.t. potentials in dimension one.

As can be seen in the proof, the loss with respect to the expected exponent is probably due to the rough energy estimates we perform and the use of the high and low-frequency results as black boxes.

A modification of the rough argument in the general case should probably allow to prove similar results for potentials $c \in L^d(M)$, for the unique continuation estimate of~\cite{LL:15} also holds for such potentials (using the rough Sobolev estimate $\|cu\|_{L^2} \leq \|c\|_{L^d} \|u\|_{H^1}$ in the proofs of that reference).

\subsection{Idea of the proof and plan of the article}
All proofs of the present paper rely on the optimal quantitative unique continuation results proved by the authors in~\cite{LL:15}. To explain the spirit of the proof, let us formulate a typical instance of this result (see~\cite[Theorem~1.1]{LL:15}) in the case $L=L^* =-\Delta$.
\begin{theorem}[Quantitative unique continuation for waves]
\label{thmobserwaveintro}
For any nonempty open subset $\omega_0$ of $M$ and any $T> T_{UC}(\omega_0)$, there exist $C, \kappa ,\mu_0>0$ such that for any $(v_0,v_1)\in H^1 (M)\times L^2(M)$ and associated solution $v \in C^0(0,T; H^1(M))$ of \eqref{e:KG-obs}, for any $\mu\geq \mu_0$, we have
\bnan
\label{e:LL15}
\nor{(v_0,v_1)}{L^2(M)\times H^{-1}(M)}\leq C e^{\kappa \mu}\nor{v}{L^2(0,T;H^1(\omega_0))}+\frac{1}{\mu}\nor{(v_0,v_1)}{H^1(M)\times L^2(M)}.
\enan
\end{theorem}
In the analytic setting, this result is a global quantitative version of the Holmgren theorem and can be proved with the theory developed by Lebeau in~\cite{Leb:Analytic}. In the $C^\infty$ case, the qualitative unique continuation result in optimal time was proved by Tataru~\cite{Tataru:95} (see also~\cite{RZ:98,Hor:97,Tataru:99} for more general operators). This followed a series of papers:~\cite{RT:72,Lerner:88} in infinite time, and then~\cite{Robbiano:91, Hormander:92}.
Concerning quantitative results, Robbiano~\cite{Robbiano:95} first proved inequality~\eqref{thmobserwaveintro} for $T$ sufficiently large and $Ce^{\kappa \mu}$ replaced by $Ce^{\kappa \mu^2}$. This was improved by Phung \cite{Phung:10} to $C_\eps e^{\kappa \mu^{1+\eps}}$, still in large time. In~\cite{Tatarunotes}, Tataru suggested a strategy to obtain $C_\eps e^{\kappa \mu^{1+\eps}}$ in optimal time (in domains without boundaries). At the same time we proved the above Theorem~\ref{thmobserwaveintro} \cite[Theorem~1.1]{LL:15}, Bosi, Kurylev and Lassas~\cite{BKL:15} obtained $C_\eps e^{\kappa \mu^{1+\eps}}$, still in domains without boundaries (but with constants uniform with respect to the operators involved, for applications to inverse problems). We refer to the introduction of~\cite{LL:15} for a more detailed discussion on this issue. One of the motivations for Theorem~\ref{thmobserwaveintro} is that it provides the cost of approximate controls for waves (see~\cite{Robbiano:95,LL:15}).

One of the advantages of this result is that it is proved via Carleman estimates and hence furnishes computable constants. In particular, a uniform version with respect to lower order terms is also furnished in~\cite{LL:15}, which we shall use for the proof of Theorem~\ref{t:dep-pot}.

\bigskip
With this in hand, the starting point of this paper is a proof of the full observability estimate~\eqref{e:general-obs} from the high-frequency one~\eqref{e:general-obs-relaxed} and~\eqref{e:LL15}. This is the following very basic observation: plugging~\eqref{e:LL15} in~\eqref{e:general-obs-relaxed} yields, for all $\mu \geq \mu_0$,
\bna
\left( 1- \frac{2\mathfrak{C}'}{\mu^2}\right)\E_1(v_0 , v_1) \leq  \mathfrak{C} \int_0^T \|\obs v(t) \|_{H^1(M)}^2 dt  +2\mathfrak{C}' C^2 e^{2\kappa \mu} \int_0^T\nor{v(t)}{H^1(\omega_0)}^2 dt .
\ena
Taking also $\mu \geq \sqrt{2 \mathfrak{C}'}$, this eventually proves~\eqref{e:general-obs} with $\mathfrak{C}_{obs} \simeq  \mathfrak{C} + \mathfrak{C}' C^2 e^{2\kappa \sqrt{2 \mathfrak{C}'}}$, provided that $\ovl{\omega_0} \subset \omega$ and $T_{UC}(\omega_0) \leq T_{GCC}(\omega)$ (which we may always assume, see Appendix~\ref{s:compare-time}).
This directly provides a quantitative treatment of Step (2): passing from the relaxed observability inequality~\eqref{e:general-obs-relaxed} to the full observability inequality~\eqref{e:general-obs}.

\bigskip
On a compact manifold without boundary, we also explain how to prove~\eqref{e:general-obs-relaxed} in a constructive way. This follows the spirit of the paper by Dehman and Lebeau~\cite{DL:09}. We write the observation as
$$
\int_0^T \|\obs v(t) \|_{H^1(M)}^2 dt  =\left( \mathscr{G}_T V_0 , V_0 \right) , \quad V_0 = (v_0 , v_1) ,
$$
where $\mathscr{G}_T$ is the Gramian operator of the control problem. As in~\cite{DL:09}, we prove essentially that $\mathscr{G}_T$ is a pseudodifferential operator of order zero with principal symbol $\sigma_0(\mathscr{G}_T) = \int_0^T \obs^2 \circ \pi  \circ\varphi_t (\rho)dt$. We have $\sigma_0(\mathscr{G}_T) \geq \mathfrak{K}(T)$ uniformly on $S^*M$; the use of the Sharp G{\aa}rding inequality then proves that $\left( \mathscr{G}_T V_0 , V_0 \right) \geq \mathfrak{K}(T) \E_1(V_0)$, modulo lower order terms $C\E_0(V_0)$, which is exactly~\eqref{e:general-obs-relaxed} with $\mathfrak{C}= \frac{1}{\mathfrak{K}(T)}$ and $\mathfrak{C}'= \frac{C}{\mathfrak{K}(T)}$.

\bigskip
The plan of the paper is the following.  Section~\ref{s:estimate-T} is devoted to the study of the limit $T \to T_{GCC}(\omega)^+$. In Section~\ref{s:prelim}, we introduce some notation used throughout the paper. Then, in Section \ref{s:model-case} we perform the high-frequency analysis of a model case, namely the Klein-Gordon equation, corresponding to $L=L^*=-\Delta+1$ (and prove in particular Theorem~\ref{t:lower-obs}). In this case, the proofs are simpler to write, so we chose to expose it separately. Then, we conclude in this case the proof of Theorem~\ref{th:KG-obs} in Section~\ref{s:low-freq-T}. Finally, we consider the general case of Theorem~\ref{t:main-result-lot} in Section~\ref{s:generalcase-T}. Only the high-frequency analysis needs care, for the low-frequency analysis is exactly that of Section~\ref{s:low-freq-T}.

Then, in Section~\ref{s:potential}, we consider the problem of uniform observation with respect to potentials. We first prove the refined low-frequency estimates in this case in Section~\ref{s:potential-lowfreq}.
Second,  we conclude the proof of Theorem~\ref{t:dep-pot} in Section~\ref{s:potential-fullest}, using as a black box the high-frequency estimates of~\cite{BLR:88, BLR:92}.

The article ends with two appendices. Appendix~\ref{s:pseudocalc} concerns general fact on pseudodifferential calculus. It contains in particular a proof of a non-autonomous non-selfadjoint Egorov theorem (Appendix~\ref{s:egorov}), of some smoothing properties of operators (Appendix~\ref{s:smoothing-prop}) and some uniform calculus estimates on compact manifolds (Appendix~\ref{s:unif-pseudo}). The second Appendix~\ref{app:geom} is devoted to geometry and contains some elementary properties of $T_{GCC}(\omega)$ and $T_{UC}(\omega)$ (Appendix~\ref{s:compare-time}).

\section{The observability constant as $T \to T_{GCC}(\omega)^+$}
\label{s:estimate-T}

In all this section, $\d M =\emptyset$. In Sections~\ref{s:model-case} and~\ref{s:low-freq-T}, we first prove Theorems~\ref{t:lower-obs} and~\ref{th:KG-obs}: in these two sections, the operator $L$ is $-\Delta +1$. In Section~\ref{s:generalcase-T}, we then prove their generalization, namely Theorem~\ref{t:main-result-lot}: in that section, $L$ has the general form given in~\eqref{e:def-op-elliptiq-adj}. 
The reason why the analysis is simpler in the Klein-Gordon case is that we have the exact factorization formula, for $\Lambda = (-\Delta+1)^\frac12$,
\bnan
\label{e:factorization}
\d_t^2   -\Delta   + 1 = \d_t^2 - \Lambda^2 = (\d_t - i \Lambda)(\d_t + i \Lambda) .
\enan
Of course, this is not needed (as shown in Section~\ref{s:generalcase-T}) but gives rise to several simplifications. We refer to Remark~\ref{rem:Lambda} concerning the use of an exact square root of $-\Delta+1$.

\subsection{Preliminaries}
\label{s:prelim}
We denote by $(e_j)_{j \in \N}$ a Hilbert basis of eigenfunctions of the Laplace-Beltrami operator, associated to the eigenvalues $(\kappa_j)_{j \in \N}$. In particular, we have $e_j \in \Cinf(M)$, $- \Delta e_j = \kappa_j e_j$, $\kappa_j \geq 0$, and $(e_j , e_k)_{L^2(M)} = \delta_{jk}$.

For $s \in \R$, we shall often use the operator $\Lambda^s = (-\Delta+1)^{\frac{s}{2}} : \Cinf(M) \to \Cinf(M)$, defined spectrally by
$$
\Lambda^s f =\sum_{j \in \N} (\kappa_j + 1)^\frac{s}{2} (f, e_j)_{L^2(M)} e_j , \quad  s\in \R .
$$
By duality, it may be extended by duality as an operator $\Lambda^s  : \D'(M) \to \D'(M)$.
We define the Sobolev spaces 
$$
 H^s(M) =  \{f \in \D'(M) ,\Lambda^s f \in L^2(M) \} , \quad  s\in \R .
$$
and associated norms
\begin{align}
\label{def-Hs}
\|f\|_{H^s(M)}^2 = \| \Lambda^s f \|_{L^2(M)}^2, \qquad 
\|(f , g)\|_{H^s(M) \times H^{\sigma}(M)}^2 = \|f \|_{H^s(M) }^2 +  \|g \|_{H^{\sigma}(M) }^2  .
\end{align}
We also sometimes write $H^s(M;\C^2) = H^s(M)\times H^s(M)$. 
On any $H^\sigma(M)$, $\sigma \in \R$, the operator $\Lambda^s$ is an unbounded selfadjoint operator with domain $H^{\sigma+s}(M)$.
In particular, $\Lambda^s$ is an isomorphism from $H^{\sigma+s}(M)$ onto $H^\sigma(M)$.

Let us also recall that, given an open set $\Omega \subset M$, we may define the local $H^1$-norm on $\Omega$ by
$$
\nor{v}{H^1(\Omega)}^2  = \int_{\Omega} |\nabla v|^2 + |v|^2 dx , \quad \text{ with } |\nabla v|^2(x) = g_x(\nabla v(x) , \nabla v(x)) ,
$$
which, in case $\Omega = M$, is equivalent to the global $H^1$-norm defined by~\eqref{def-Hs}.

We shall also use the energy-spaces $\H^s(M) = H^s (M) \times H^{s-1}(M)$ associated to the energy norms
$$
\|(v_0 , v_1)\|_{\H^s(M)}^2 := \|v_0\|_{H^s}^2 + \| v_1 \|_{H^{s-1}}^2  , \qquad \E_s(v_0 , v_1) := \frac12 \|(v_0 , v_1)\|_{\H^s(M)}^2 .
$$

According to \cite{Seeley:66} (or \cite[Theorem~11.2]{Shubin:01}), we have
$$ë
\Lambda^s \in \Psi_{\phg}^{s} (M) , \quad \text{with} \quad 
\sigma_s(\Lambda^s)(x, \xi) : = \lambda^s(x, \xi)= |\xi|_x^s, \quad (x, \xi) \in T^*M \setminus 0 ,
$$
where all notations are defined in Appendix~\ref{app:geom}.
We denote by $(e^{\pm it\Lambda})_{t \in \R}$ the group of operators acting on $H^s(M)$ generated by $ \pm i \Lambda$. 

We denote by $\varphi_t = \varphi^+_t$ (both notations will be used) the hamiltonian flow of $\lambda(x, \xi)= |\xi|_x$ on $T^*M\setminus 0$, and $\varphi^-_t$ that of $-\lambda$. They are linked by $\varphi^-_t =\varphi^+_{-t} $, according to Lemma~\ref{l:phi+phi-}, but is is convenient to keep two different notations.

We conclude this notation section with the following definition.
\begin{definition}
\label{def:Bo}
Assume we are given $I= \mathcal{I}_1 \times \cdots \times \mathcal{I}_N$ a product of intervals of $\R$ (possibly reduced to a single interval) and $S$ an application from $I$ with value in the set of bounded linear operators acting from a Banach space $B_1$ to another one $B_2$.  We shall say that $S \in \Bo( I ;\L(B_1; B_2))$ if 
\begin{enumerate}
\item there exists $C>0$ such that $\nor{S(t)u}{B_2}\leq C \nor{u}{B_1}$ for any $u\in B_1$ and $t\in I$;
\item for any $j \in \{1, , \cdots , N\}$ and any $(t_1, \cdots , t_{j-1}, t_{j+1}, \cdots , t_N) \in  \mathcal{I}_1 \times \cdots \times \mathcal{I}_{j-1} \times  \mathcal{I}_{j+1} \times \cdots \times \mathcal{I}_N$, the map $t_j \to S(t_1 , \cdots , t_N)u$ is in $\Con^0(\mathcal{I}_j ; B_2)$ for any $u \in B_1$.
\end{enumerate}Similarly, we write $S \in \Bo_{\loc}( I ;\L(B_1; B_2))$ if this estimate is satisfied on any compact set of $I$.
\end{definition}
In the applications, we always have $I \subset \R$ of $I = \mathcal{I} \times \mathcal{I}$ with $\mathcal{I}$ an interval of $\R$, in particular when studying the solution operator associated to a strictly hyperbolic Cauchy problem, see Appendix~\ref{s:pseudo-elementaire}.

Note that if $S \in \Bo( I ;\L(B_1; B_2))$ and $T \in \Bo( I ;\L(B_2; B_3))$, then we have $TS \in \Bo( I ;\L(B_1; B_3))$.

Note also that the space $\Bo( I ;\L(B_1; B_2))$ is not included in $L^{\infty}(I;\L(B_1; B_2))$, for maps in $S \in \Bo( I ;\L(B_1; B_2))$ are not a priori measurable in the Bochner sense. However, for all $u \in B_1$ and $(t_1, \cdots , t_{j-1}, t_{j+1}, \cdots , t_N) \in  \mathcal{I}_1 \times \cdots \times \mathcal{I}_{j-1} \times  \mathcal{I}_{j+1} \times \cdots \times \mathcal{I}_N$ fixed, the partial map $t_j \to S(t_1 , \cdots , t_N)u$ is in $\Con^0(\mathcal{I}_j ; B_2)$ and hence (Bochner) integrable. With a usual abuse of notation, for $T_j \in \mathcal{I}_j$ (and assume $0 \in \mathcal{I}_j$), we shall write $\int_0^{T_j} S(t_1 , \cdots , t_N) dt_j$ the linear map
$$
 u \mapsto \int_0^{T_j} \big(S(t_1 , \cdots , t_N) u \big)dt_j .
$$
Remark then that $(t_1, \cdots , t_{j-1}, T_j , t_{j+1}, \cdots , t_N) \mapsto \int_0^{T_j} S(t_1 , \cdots , t_N) dt_j$ belongs to $\Bo_{\loc}( I ;\L(B_1; B_2))$ if $S$ does. 
Also, we have $\Con^0(I ; \Psi^m_{\phg}(M)) \subset \Bo_{\loc}( I ;\L(H^\sigma(M); H^{\sigma - m}(M)))$, according to Corollary~\ref{cor:unif-bound}.

These facts will be used throughout the section.

\subsection{The high-frequency estimate for the Klein-Gordon equation}
\label{s:model-case}

In the present case of the Klein-Gordon equation, that is~\eqref{e:KG-obs} with $L^* = -\Delta+1$, and in view of the factorization formula~\eqref{e:factorization}, we use the following splitting:
\begin{align}
\label{eq: splitting g}
v_+ = \frac12 \big(  v_0 - i \Lambda^{-1} v_1 \big)  , \qquad  v_- = \frac12 \big(  v_0 + i \Lambda^{-1} v_1 \big) ,
\end{align}
so that 
$$
v_0 = v_+ + v_- , \qquad 
v_1 = i \Lambda(v_+ - v_-)  .
$$

we denote by $\Sigma$ the {\em isomorphism} corresponding to the splitting~\eqref{eq: splitting g}:
\begin{equation*}
\begin{array}{rccc}
\Sigma : & H^s(M) \times H^{s-1}(M)&  \to & H^s(M ) \times H^s(M )\\
&(v_0,v_1) &\mapsto & (v_+ , v_- )  .
\end{array}
\end{equation*}
that is 
\begin{equation}
\label{e:splitting-sigma}
\Sigma  = \frac12\left(
\begin{array}{cc}
1  &- i \Lambda^{-1}  \\
1 & i \Lambda^{-1}
\end{array}
\right) , \qquad 
\Sigma^{-1}  =  \left(
\begin{array}{cc}
1  & 1 \\
i \Lambda & -i \Lambda
\end{array}
\right) .
\end{equation}

Notice that the operator $\Sigma$ is (almost) an {\em isometry} $H^s(M) \times H^{s-1}(M) \to H^s(M) \times H^s(M)$. Indeed, if $(v_+ , v_- ) = \Sigma (v_0,v_1)$, we have
\begin{equation}
\label{e:isometry}
\|(v_0,v_1) \|_{H^s(M) \times H^{s-1}(M)}^2 = \|v_+ + v_- \|_{H^s(M)}^2 + \|v_+ - v_- \|_{H^s(M)}^2 =2 \left( \|v_+   \|_{H^s(M)}^2 + \| v_- \|_{H^s(M)}^2 \right),
\end{equation}
that is
\begin{equation}
\label{e:isometry-energy}
\|v_+   \|_{H^s(M)}^2 + \| v_- \|_{H^s(M)}^2 = \E_s(v_0 ,v_1)= \E_s(\Sigma^{-1}(v_+ , v_-)).
\end{equation}
According to~\eqref{e:factorization}, the expression of the solution of System~\eqref{e:KG-obs} is simply
\begin{align}
\label{eq: HUM Duhamel splitting 1}
v(t) = e^{i t\Lambda} v_+ + e^{- i t\Lambda} v_- .
\end{align}

We can now recall a result of \cite{DL:09} (in a little different context), providing a characterization of the Gramian operator (in the wave splitting~\eqref{eq: splitting g}).

\begin{proposition}
\label{th: HUM operator-waves}
Denoting by $V_0 = (v_0,  v_1) \in H^s(M) \times H^{s-1}(M)$ the initial data for System~\eqref{e:KG-obs}, we have 
\begin{equation}
\label{eq: obs = G_T}
\int_0^T \|\obs v(t)\|_{H^s(M)}^2 dt  = \big(\G_T \Sigma V_0, \Sigma V_0 \big)_{H^s(M) \times H^s(M)} ,
\end{equation}
where
\begin{equation}
\label{eq: G_T}
\G_T = \int_0^T  \left(
\begin{array}{cc}
  e^{-i t \Lambda} B e^{i t \Lambda}  &  e^{-i t \Lambda} B  e^{- i t \Lambda}    \\
  e^{i t \Lambda}  B  e^{i t \Lambda}  & e^{i t \Lambda}  B  e^{- i t \Lambda} 
\end{array}
\right)  dt , \qquad B = \Lambda^{-2s} \obs \Lambda^{2s} \obs .
\end{equation}
Moreover, the operator $\G_T$ can be decomposed as $\G_T = G_T + R_T$ with 
$$
R_T \in \Bo_{\loc} (\R^+; \L(H^\sigma (M;\C^2) ;H^{\sigma +1}(M;\C^2)) , \quad \text{ for all } \sigma \in \R ,
$$
 and $G_T  \in \Cinf ( \R_T ;\Psi_{\phg}^0(M ;\C^{2\times2}))$ has principal symbol
\begin{equation}
\label{eq: sigma_0 (G_T)}
\sigma_0 (G_T) = \left(
\begin{array}{cc}
\int_0^T \obs^2 \circ \varphi_t^- dt  & 0  \\
0 & \int_0^T \obs^2 \circ \varphi_t^+ dt 
\end{array}
\right) \in S_{\phg}^{0}(T^*M , \C^{2 \times 2}). 
\end{equation}
\end{proposition}
Note that the Gramian operator $\G_T$ actually depends on the space $H^s(M)$ (even not written in the notation). An interesting fact is that its principal symbol does not depend on $s$.
The result of Proposition~\ref{th: HUM operator-waves} is essentially proved in \cite[Section~4.1]{DL:09} and we reproduce a proof below for the sake of completeness.

\begin{remark}
Note that the operator $B = \Lambda^{-2s} \obs \Lambda^{2s} \obs$ is symmetric on $H^s(M)$ since
$$
(Bg,h)_{H^s(M)}=  \big(\Lambda^s(\Lambda^{-2s} \obs \Lambda^{2s} \obs) g,\Lambda^sh \big)_{L^2 (M)}= (\obs g, \obs h)_{H^s(M)} , \qquad g,h \in H^s(M).
$$
\end{remark}

\begin{proof}[Proof of Proposition~\ref{th: HUM operator-waves}]
We write $\Sigma V_0= (v_+ , v_- )$, $v (t) =  e^{i t \Lambda} v_+ + e^{-i t \Lambda} v_- $ the associated solution, and develop the inner product
\begin{align}
\int_0^T \|\obs v (t)\|_{H^s(M)}^2 dt &= \int_0^T \left( \Lambda^s \obs( e^{i t \Lambda} v_+ + e^{-i t \Lambda} v_-  ),  \Lambda^s \obs( e^{i t \Lambda} v_+ + e^{-i t \Lambda} v_- ) \right)_{L^2(M)} dt \\
&= \int_0^T \left( \Lambda^{- 2s} \obs \Lambda^{2s} \obs( e^{i t \Lambda} v_+ + e^{-i t \Lambda} v_-  ), ( e^{i t \Lambda} v_+ + e^{-i t \Lambda} v_- ) \right)_{H^s(M)} dt .
\end{align}
This directly yields the sought form for the operator $\G_T$ given by~\eqref{eq: G_T}. 
The Egorov Theorem~\ref{theorem: Egorov} (see also Remark \ref{rkegorovsimple}) in the Appendix then implies that 
\begin{equation}
 \left(
\begin{array}{cc}
  \int_0^T  e^{-i t \Lambda} B e^{i t \Lambda}  dt&  0   \\
0  &  \int_0^T e^{i t \Lambda}  B  e^{- i t \Lambda} dt 
\end{array}
\right)   = G_T + R_T^0 ,
\end{equation}
with $G_T \in \Cinf ( \R_T ;\Psi_{\phg}^0(M ;\C^{2\times2}))$ has principal symbol given by~\eqref{eq: sigma_0 (G_T)} and $R_T^0 \in \Bo_{\loc} (\R^+; \L(H^\sigma (M;\C^2) ;H^{\sigma +1}(M;\C^2))$ for all $\sigma \in \R$. Finally, Lemma~\ref{lemma: regularity Hs} implies that 
\begin{equation}
R_T^1 =  \left(
\begin{array}{cc}
0 &   \int_0^T  e^{-i t \Lambda} B e^{- i t \Lambda}  dt   \\
\int_0^T e^{i t \Lambda}  B  e^{i t \Lambda} dt   &  0
\end{array}
\right)   ,
\end{equation}
is also in $\Bo_{\loc} (\R^+; \L(H^\sigma (M) ;H^{\sigma +1}(M))$ for all $\sigma \in \R$, which concludes the proof with $R_T = R_T^0 +R_T^1$.
\end{proof}

As a first consequence of Proposition~\ref{th: HUM operator-waves}, we deduce a proof of Theorem~\ref{t:lower-obs}.

\bnp[Proof of Theorem~\ref{t:lower-obs}]
Let $\rho_0 = (x_0 , \xi_0)\in S^*M$ that realizes the minimum in~\eqref{e:def-KT}, that is,
\bnan
\label{e:choice-rho0}
\mathfrak{K}(T) = \min_{\rho \in S^*M} \int_0^T \obs^2 \circ \pi \circ \varphi_t (\rho)dt =  \int_0^T \obs^2 \circ \pi \circ  \varphi_t (\rho_0)dt .
\enan
Take a local chart $(U_\kappa , \kappa)$ of $M$ such
that $x_0 \in U_\kappa$. We denote by $(y_0 , \eta_0)$ the coordinates
of $\rho_0$ in this chart.
We choose $\psi \in \Cinfc(\R^n)$ such that $\supp(\psi) \subset
\kappa(U_\kappa)$, and $\psi = 1$ in a \nhd of $y_0$. Next we define
$$
w^k (y) = C_0 k^{\frac{n}{4}} e^{i k \varphi(y)}\psi(y), \quad  
\text{with } 
\varphi(y) = y \cdot \eta_0 + i(y-y_0)^2 \text{ and } C_0 >0 .
$$
Setting now 
\begin{align}
\label{eq: sequence concentrating}
v_+^k = \Lambda^{-s} \kappa^* w^k \in \Cinfc(M), 
\end{align}
we have $v_+^k \rightharpoonup 0$ in $H^s(M)$, $\lim_{k \to
  \infty} \|v_+^k\|_{H^s(M)} = 1$ for an appropriate choice of
$C_0$. Moreover, a classical computation on $(w^k)_{k \in \N}$ shows
that $(v_+^k)_{k \in \N}$ satisfies
\begin{equation}
\label{e:Amdm}
\left(A v_+^k , v_+^k \right)_{H^{s}(M)} \to \sigma_{0}(A)(\rho_0) ,  \quad \text{ for all } A \in \Psi_{\phg}^0(M) .
\end{equation}
Next, we set $v_-^k = 0$ for all $k \in \N$, and $V^k = \Sigma^{-1} (v_+^k , v_-^k)  \in H^s(M) \times H^{s-1}(M)$, so that $\E_s(V^k) \to 1$ as $k \to \infty$.
Applying now~\eqref{eq: obs = G_T} to $V^k$, we have  
\bna
\int_0^T \|\obs v^k (t)\|_{H^s(M)}^2 dt  = \big(\G_T \Sigma V^k, \Sigma V^k \big)_{H^s(M) \times H^s(M)} ,
 \ena
where $v^k(t)$ is the solution to System~\eqref{e:KG-obs} with initial data $V^k$.
Proposition~\ref{th: HUM operator-waves} and~\eqref{e:Amdm} also imply
\begin{align*}
\lim_{k \to \infty}
 \big(\G_T \Sigma V^k, \Sigma V^k \big)_{H^s(M) \times H^s(M)} 
&= \lim_{k \to \infty}
 \big((G_T + R_T) \Sigma V^k, \Sigma V^k \big)_{H^s(M) \times H^s(M)} \\
 &= \lim_{k \to \infty}
 \big( G_T  \Sigma V^k, \Sigma V^k \big)_{H^s(M) \times H^s(M)} \\
& = \int_0^T \obs^2 \circ \pi \circ \varphi_t (\rho_0) dt = \mathfrak{K}(T) ,
\end{align*}
where we used that $R_T$ is $1$-smoothing, that $G_T \in \Psi_{\phg}^0(M)$ has principal symbol given by~\eqref{eq: sigma_0 (G_T)}, and the choice of $\rho_0$ in~\eqref{e:choice-rho0}.
Finally using the assumed observability estimate~\eqref{e:general-obs} with $V^k$, and taking the limit $k \to \infty$ yields
$$
 \mathfrak{C}_{obs}(T) \mathfrak{K}(T)  \leftarrow  \mathfrak{C}_{obs}(T)  \int_0^T \|\obs v^k(t) \|_{H^s(M)}^2 dt \geq \E_s(V^k) \to 1 .
$$
This implies $\mathfrak{C}_{obs}(T) \geq \mathfrak{K}(T)^{-1}$, and concludes the proof of Theorem~\ref{t:lower-obs}.
\enp
\begin{remark}
Note that~\eqref{e:Amdm} translates the fact that the sequence $(v_+^k)_{k \in \N}$ is a pure sequence admitting the $H^s$-microlocal defect
measure $\delta_{\rho =\rho_0}$ in the sense of~\cite{Gerard:91,Tartar:90}. Similarly, the $H^s$-microlocal defect
measure of the sequence $(V^k)_{k \in \N}$ is 
\begin{align*}
\mu =
\left(
\begin{array}{cc}
\delta_{\rho_0} &0 \\
0&0 \\
\end{array}
\right) .
\end{align*}
\end{remark}

As a second consequence of Proposition~\ref{th: HUM operator-waves}, we also obtain the following high-frequency observability inequality.

\begin{proposition}
\label{l:obsHF}
For any $T_0>0$, there exists a constant $C_0>0$ such that for all $T \in [0,T_0]$, for all $V_0 =(v_0, v_1) \in H^s(M)\times H^{s-1}(M)$ and associated solution $v$ of~\eqref{e:KG-obs}, we have
\begin{equation}
\label{eq: obs HF}
\int_0^T \|\obs v (t)\|_{H^s(M)}^2 dt  \geq \mathfrak{K}(T)  \E_s(V_0) -C_0 \E_{s-1/2}(V_0),
\end{equation}
where $\mathfrak{K}(T)$ is defined by~\eqref{e:def-KT} and $L^*=-\Delta+1$.
\end{proposition}

\begin{proof}[Proof of Proposition~\ref{l:obsHF}]
We first write $\Sigma V_0 = (v_+ , v_-  ) = V$, and use \eqref{eq: obs = G_T}. We have 
\begin{equation}
\label{e:decomp1}
\big(\G_T V, V  \big)_{H^s(M ; \C^2)} = \big(G_T V, V  \big)_{H^s(M ; \C^2)}  + \big( R_T V, V  \big)_{H^s(M ; \C^2)} .
\end{equation}
Using that $R_T \in \Bo_{\loc} (\R^+; \L(H^{s-1/2} (M;\C^2) ;H^{s+1/2}(M;\C^2))$, we have
\begin{equation}
\label{e:decomp2}
\big( R_T V, V  \big)_{H^s(M;\C^2)} \leq  \| R_T V \|_{H^{s+ 1/2}(M;\C^2)}  \| V \|_{H^{s-1/2}(M;\C^2)} 
\leq C_T   \| V \|_{H^{s-1/2}(M;\C^2)}^2 ,
\end{equation}
where $C_T$ is bounded on compact time intervals.

Next, according to~\eqref{eq: sigma_0 (G_T)}, the principal symbol of the operator $G_T  - \mathfrak{K}(T)\id \in \Psi_{\phg}^0(M; \C^{2\times2})$ is
$$
\sigma_0(G_T - \mathfrak{K} (T)\id) = 
 \left(
\begin{array}{cc}
\int_0^T \obs^2 \circ \varphi_t^- dt - \mathfrak{K}(T)  & 0  \\
0 & \int_0^T \obs^2 \circ \varphi_t^+ dt  - \mathfrak{K} (T)
\end{array}
\right) ,
$$
which is diagonal with nonnegative components since, according to Corollary~\ref{e:moy-=moy+}, we have 
$$
\mathfrak{K}(T) = \min_{\rho \in S^*M} \int_0^T \obs^2 \circ \varphi_t (\rho) dt =  \min_{\rho \in S^*M} \int_0^T \obs^2 \circ \varphi_t^- (\rho) dt .
$$
Using the G{\aa}rding inequality of Theorem~\ref{th: sharp garding manifold} gives the existence of $C>0$ such that, for all $V \in H^s(M ; \C^2)$ all $T \in [0,T_0]$,
\begin{align}
\label{eq: garding hum 2}
\left((G_T - \mathfrak{K}(T)\id) V , V \right)_{H^s(M ; \C^2)} \geq - C \|V\|_{H^{s-1/2}(M; \C^2)}^2 .
\end{align}
Combining~\eqref{e:decomp1}, \eqref{e:decomp2} and \eqref{eq: garding hum 2} now yields the existence of $C>0$ such that, for all $V \in H^s(M ; \C^2)$ all $T \in [0,T_0]$,
$$
\big(\G_T V, V  \big)_{H^s(M) \times H^s(M)} \geq 
\mathfrak{K}(T) \|V \|_{H^s(M ; \C^2)}^2  - C \|V\|_{H^{s-1/2}(M; \C^2)}^2 
$$
Recalling \eqref{e:isometry-energy} that $\|V\|_{H^\sigma(M ; \C^2)}^2 = \E_\sigma(V_0)$ concludes the proof of~\eqref{eq: obs HF}.
\end{proof}

\bigskip
To conclude this section, we explain the terminology ``high-frequency observability estimates''. 
Let first $T_0 > T_{GCC}(\omega)$ be fixed and denote by $C_0>0$ the associated constant given by Proposition~\ref{l:obsHF}. We define the following $T$-dependent subset of $\H^s$ by
\bna
\H^s_{HF}(T)=\left\{V_0\in \H^s , \quad \E_{s-\frac12}(V_0) \leq  \frac{\mathfrak{K}(T)}{4C_0}  \E_{s}(V_0) \right\}.
\ena
Note that this space is nonlinear, however homogeneous, in the sense that $V_0 \in \H^s_{HF}(T) \implies \R V_0 \in \H^s_{HF}(T)$.
Remark also that $\H^s_{HF}(T)=\left\{0\right\}$ if $T\leq T_{GCC}(\omega)$, since $\mathfrak{K}(T)=0$ in this case.
We may now formulate an immediate corollary of Proposition~\ref{l:obsHF}, only consisting in a rewriting of that statement for data in $\H^s_{HF}(T)$, yielding a full observability inequality.
\begin{corollary}
\label{cor:obsHFbis}
For all $V_0 =(v_0, v_1) \in \H^s_{HF}(T)$ and associated solution $v$ of~\eqref{e:KG-obs}, we have
\begin{equation}
\label{eq: obs HFbis}
\int_0^T \|\obs v (t)\|_{H^s(M)}^2 dt  \geq  \frac{\mathfrak{K}(T)}{2} \E_s(V_0) .
\end{equation}
\end{corollary}
Finally, the following Lemma states that data spectrally supported at high-frequency (in terms of the spectral theory of $-\Delta$) are in $\H^s_{HF}(T)$. As such, they satisfy the full observability inequality~\eqref{eq: obs HFbis}.
\begin{lemma}
\label{lm:obsHFbis}
Denoting by 
$$
F_{\kappa}^s = \left\{V_0 \in \H^s , \Pi_\kappa V_0 = 0 \right\} , \quad \text{with} \quad \Pi_\kappa(v_0, v_1) = \left( \sum_{\kappa_j \leq \kappa} (v_0 , e_j)_{L^2(M)} e_j, \sum_{\kappa_j \leq \kappa} (v_1 , e_j)_{L^2(M)} e_j\right) ,
$$
we have
\bna
\kappa \geq \left(  \frac{4 C_0}{\mathfrak{K}(T)} \right)^2 - 1 \quad \Longrightarrow \quad  F_{\kappa}^s \subset \H^s_{HF}(T) .
\ena
\end{lemma}
When doing this, notice that we compare the typical frequency $\kappa^\frac12$ to the blow up of the observation $\mathfrak{K}(T)^{-1}$. We recall that $ \left(\frac{4C_0}{\mathfrak{K}(T)}\right)^2\underset{T\rightarrow T_{GCC}^+(\omega)}{\longrightarrow} +\infty$. 

\begin{proof}
If $(u, v) \in F_{\kappa}^s$, then we have $\Pi_\kappa(u, v)=0$ so that, with $u_j = (u , e_j)_{L^2(M)}$ and $v_j = (v , e_j)_{L^2(M)}$, we obtain
\bna
2 \E_{s-\frac12}(u,v) &= & \sum_{\kappa_j > \kappa} (\kappa_j + 1)^{s-\frac12} |u_j|^2 + (\kappa_j + 1)^{s-\frac32} |v_j|^2 \\
& \leq & (\kappa  + 1)^{-\frac12} \sum_{\kappa_j > \kappa} (\kappa_j + 1)^{s } |u_j|^2 + (\kappa_j + 1)^{s-1} |v_j|^2 \\
& \leq & (\kappa  + 1)^{-\frac12} 2 \E_{s}(u,v) .
\ena
If now $\kappa +1 \geq \left(  \frac{4 C_0}{\mathfrak{K}(T)} \right)^2$, this directly implies $(u,v) \in \H^s_{HF}(T)$.
\end{proof}

\subsection{The full observability estimate}
\label{s:low-freq-T}
Once the high-frequency observability estimate is proved, it remains to say something on the low-frequencies, i.e. remove the term
$\E_{s-1/2}(V_0)$ in the right hand-side of~\eqref{eq: obs HF} for general data (as opposed to the result in Corollary~\ref{cor:obsHFbis}). This is based on~\cite{LL:15}. We only use the case $s=1$ in~\eqref{eq: obs HF} to which~\cite{LL:15} is more adapted.
As a corollary of Theorem~\ref{thmobserwaveintro} (i.e.~\cite[Theorem~1.1]{LL:15}), we have the following intermediate estimates.

 \begin{corollary}
 Let $\omega_0$ be an open set of $M$ and fix $T_0 > T_{UC}(\omega_0)$. Then, there exist $\kappa ,\mu_0>0$ such that 
for any $s\in [0,1)$ there is $C>0$ such that for all $(v_0,v_1)\in \H^1 (M)$ and associated solution $v \in \Con^0(0,T; H^1(M))$ of \eqref{e:KG-obs}, for any $T\geq T_0$ and $\mu\geq \mu_0^{1-s}$, we have 
\bna
\nor{(v_0,v_1)}{\H^{s}(M)}\leq C C_s(\mu) \nor{v}{L^2(0,T;H^1(\omega_0))}+\frac{C}{\mu}\nor{(v_0,v_1)}{\H^1(M)},   
\ena
with $C_s(\mu) = \mu^{\frac{s}{1-s}} e^{\kappa \mu^{\frac{1}{1-s}}}$. 
In particular, for any $T\geq T_0$ and $\mu\geq \mu_0^{1/2}$, we have
\bnan
\label{observweak12}
\nor{(v_0,v_1)}{\H^{1/2}(M)} \leq C\mu e^{\kappa \mu^2}\nor{v}{L^2(0,T;H^1(\omega_0))}+\frac{C}{\mu} \nor{(v_0,v_1)}{\H^1(M)}.
\enan
\end{corollary}
 \bnp
 We denote by $V_0 = (v_0 , v_1)$ all along the proof.
 Using an interpolation estimate and Young inequality, with $\eta>0$, we obtain for $C>0$ (depending on $s$) 
 \bna
\nor{V_0}{\H^{s}(M)} \leq C \nor{V_0}{\H^0(M)}^{1-s}\nor{V_0}{\H^1(M)}^{s} \leq C (1-s)\eta^{-1/(1-s)}\nor{V_0}{\H^0(M)}+ Cs\eta^{1/s}\nor{V_0}{\H^1(M)} .
\ena
Then using~\eqref{e:LL15} for $T = T_0 > T_{UC}(\omega_0)$ yields, for $\mu \geq \mu_0$,
\bna
\nor{V_0}{\H^s(M)}  \leq C \eta^{-1/(1-s)}\left[C_0(\mu) \nor{v}{L^2(0,T_0;H^1(\omega_0))} + \frac{1}{\mu} \nor{V_0}{\H^1(M)}\right]+C\eta^{1/s}\nor{V_0}{\H^1(M)} 
 \ena
Now, we take $\eta$ such that $\eta^{\frac{1}{s(1-s)}}=\frac{1}{\mu}$, implying, for $\mu \geq \mu_0$,
 \bna
\nor{V_0}{\H^s(M)} \leq C  \mu^{s}C_0(\mu) \nor{v}{L^2(0,T_0;H^1(\omega_0))} + \frac{C}{\mu^{1-s}}\nor{V_0}{\H^1(M)}.
 \ena 
 Finally, writing $\tilde{\mu}=\mu^{1-s}$, there is $C>0$ (depending on $s$) such that for any $\tilde{\mu} \geq \mu_0^{1-s}$, we have
  \bna
\nor{V_0}{\H^s(M)} \leq C \tilde{\mu}^{\frac{s}{1-s}}C_0(\tilde{\mu}^{\frac{1}{1-s}})\nor{v}{L^2(0,T_0;H^1(\omega_0))} + \frac{C}{\tilde\mu} \nor{V_0}{\H^1(M)} .
 \ena
 Using that $\nor{v}{L^2(0,T_0;H^1(\omega_0))} \leq \nor{v}{L^2(0,T;H^1(\omega_0))} $ for all $T\geq T_0$ concludes the proof of the corollary.
 \enp
 
 We can now conclude the proof of the main theorem in the model case of the Klein Gordon equation by combining the high-frequency estimate~\eqref{eq: obs HF} and the low-frequency estimate~\eqref{observweak12}.

\bnp[Proof of the observability Theorem~\ref{th:KG-obs}]
First, according to Lemma~\ref{l:omega-0} and the assumption $T_{UC}(\omega) < T_{GCC}(\omega)$, there is an open subset $\omega_0$ of $M$ such that 
$$
\ovl{\omega}_0 \subset \omega ,  \quad \text{and} \quad T_{UC}(\omega_0) < T_{GCC}(\omega) .
$$
We now choose $T_0$, so that we have
$$
0< T_{UC}(\omega)  \leq T_{UC}(\omega_0) < T_0 < T_{GCC}(\omega) < T_1
$$
(note that the assumption $T_{UC}(\omega) < T_{GCC}(\omega)$ implies $T_{GCC}(\omega)>0$ and hence $\ovl{\omega} \neq M$ and hence $T_{UC}(\omega_0) \geq T_{UC}(\omega)  >0$).

The high-frequency estimate~\eqref{eq: obs HF} for $s=1$, yields the existence of $C_0>0$ such that for all $T \in [0, T_1]$, $V_0 = (v_0,v_1)$, and associated solution $v \in \Con^0(0,T; H^1(M))$ of \eqref{e:KG-obs}, we have
\begin{equation}
\label{eq: obs HF bis}
\int_0^T \|\obs v (t)\|_{H^1(M)}^2 dt  \geq \mathfrak{K}(T) \E_1(V_0) - C_0 \E_{1/2}(V_0) .
\end{equation}
The low-frequency estimate~\eqref{observweak12} (squared) gives the existence of $C,\kappa, \mu_0 >0$, such that one has
\bnan
\label{observweak12bis}
 \E_{1/2}(V_0) \leq C\mu^2 e^{2\kappa \mu^2} \int_0^T\nor{v}{H^1(\omega_0)}^2 dt+\frac{C}{\mu^2} \E_1(V_0) ,
\enan
for all $\mu \geq \mu_0^\frac12$ and all $T\geq T_0$. These last two estimates yield, for any $\mu \geq \mu_0$ and $T \in [T_0 ,T_1]$ (the constant $C>0$ may change from line to line, but remains uniform with respect to the parameters $T$ and $\mu$),
\bna
\int_0^T \|\obs v (t)\|_{H^1(M)}^2 dt  \geq \mathfrak{K}(T) \E_1(V_0) - C \left( e^{\tilde{\kappa} \mu } \int_0^T\nor{v}{H^1(\omega_0)}^2 dt+\frac{1}{\mu} \E_1(V_0) \right),
\ena
that is 
\bna
\int_0^T \|\obs v (t)\|_{H^1(M)}^2 dt  + Ce^{\tilde{\kappa} \mu } \int_0^T\nor{v}{H^1(\omega_0)}^2 dt
\geq \left(\mathfrak{K}(T)- \frac{C}{\mu} \right) \E_1(V_0) .
\ena
Assuming now that $T>T_{GCC}(\omega)$, we have $\mathfrak{K}(T) >0$, and may choose $\mu = \max \left\{ \frac{2C}{\mathfrak{K}(T)}, \mu_0 \right\}$ to obtain, for some $\kappa^*>0 , C>0$
\bnan
\label{eq: obs HF ter}
\int_0^T \|\obs v (t)\|_{H^1(M)}^2 dt  + Ce^{\kappa^* \max \{\mathfrak{K}(T)^{-1} , \mu_0 \} } \int_0^T\nor{v}{H^1(\omega_0)}^2 dt 
\geq  \frac{\mathfrak{K}(T)}{2}  \E_1(V_0) .
\enan
Note that until this point, we did not use the assumptions on the relative location of the sets $\omega_0$ and $\omega = \{ \obs \neq 0\}$ (except that $T_{UC}(\omega_0) < T_{GCC}(\omega)$).

Finally, using that $\overline{\omega}_0 \subset \omega = \{ \obs \neq 0\}$, we have $|\obs| \geq c_0^{-1}>0$ on $\omega_0$ and $1 = \frac{\obs}{\obs}$ on this set, so that
\bna
\nor{v}{H^1(\omega_0)}^2 & =& \int_{\omega_0} |\nabla v|^2 + |v|^2 dx = 
\int_{\omega_0} |\nabla  (\obs^{-1} \obs v)|^2 + |\obs^{-1} \obs v|^2 dx \\
&\leq & \int_{\omega_0} |\nabla \obs^{-1}|^2 | \obs v|^2 + c_0^2 \int_{\omega_0} |\nabla  ( \obs v)|^2 + | \obs v|^2 dx \\
&\leq &C \int_{\omega_0} |\nabla ( \obs v)|^2 + | \obs v|^2 dx  \leq C \nor{ \obs v}{H^1(M)}^2 .
\ena
As a consequence, coming back to~\eqref{eq: obs HF ter}, there is $C, \kappa' >0$ such that for all $T \in (T_{GCC}(\omega) , T_1]$, we have
\bna
 Ce^{\kappa' \mathfrak{K}(T)^{-1} } \int_0^T \|\obs v (t)\|_{H^1(M)}^2 dt  
\geq   \E_1(V_0) ,
\ena
which concludes the proof of Theorem~\ref{th:KG-obs}.
\enp

\subsection{The high-frequency estimate in the general case}
\label{s:generalcase-T}
We now consider the general case of the observability problem~\eqref{e:KG-obs}-\eqref{e:def-op-elliptiq-adj} (dual to the controllability problem~\eqref{e:KG-cont}-\eqref{e:def-op-elliptiq}), and give a proof of Theorem~\ref{t:main-result-lot}.
We only provide below the high-frequency part of the analysis. 
The analogue of Theorem~\ref{t:lower-obs} (the lower bound) directly  follows (and does not require the analyticity of the coefficients). 
Concerning the analogue of Theorem~\ref{th:KG-obs} (the upper bound), its low-frequency part uses~\cite[Theorem~6.1]{LL:15} (instead of Theorem~\ref{thmobserwaveintro} which only deals with $L^*=-\Delta$), which requires the coefficients to be analytic in time. The proof of the full observability estimate from the high-frequency one then follows Section~\ref{s:low-freq-T}, without any modification. 

The main purpose of the following subsection is therefore the proof of Proposition \ref{l:obsHF-general} below, which is the generalization of Proposition \ref{l:obsHF}. As in the Klein-Gordon case, the proof proceeds in several steps:
\begin{itemize}
\item Writing the equation as a $2\times 2$ system.  
\item Using a trick due to Taylor to eliminate the anti-diagonal lower order terms, this is the object of Proposition~\ref{prop:representation}.
\item Applying an Egorov theorem to get a nice pseudodifferential representation. This is Proposition~\ref{th: HUM operator-waves-total}.
\item Concluding by the G{\aa}rding inequality.
\end{itemize}
\bigskip
When performing the high-frequency analysis of this observation problem, it is convenient to recast it in a more general framework. More precisely, given a fixed time $T_0>0$, we shall study the HUM control operator for the problem 
\begin{equation}
\label{e:wave-HF-obs}
\begin{cases}
\d_t^2 v -\Delta v + v + A_0 D_t v + A_1 v  = 0 ,  \quad \text{on } [0,T_0] \times M\\
(v(0), \d_t v(0)) = (v_0, v_1)
\end{cases}
\end{equation}
where $A_0 \in \Cinf( 0,T_0 ; \Psi_{\phg}^0(M))$, $D_t = \frac{\d_t}{i}$ and $A_1 \in \Cinf( 0,T_0  ; \Psi_{\phg}^1(M))$ have symbols 
$$
a_0 = \sigma_0(A_0) \in  \Cinf(  0,T_0 ; S_{\phg}^0(M)) , \quad a_1 = \sigma_1(A_1) \in  \Cinf(  0,T_0  ; S_{\phg}^1(M))  .
$$
The main additional difficulty with respect to the model case of Section~\ref{s:model-case} is that we do not have the simple representation formula~\eqref{eq: HUM Duhamel splitting 1} for the solution.

\medskip
The equation \eqref{e:KG-obs}-\eqref{e:def-op-elliptiq-adj} under interest is a particular case of~\eqref{e:wave-HF-obs} with 
\bnan
\label{e:lot-link-a-b-1}
A_0 =- i \ovl{b_0}, \quad \text{with }
 a_0 (t, x, \xi)= - i \ovl{b_0}(t,x),\\
 \label{e:lot-link-a-b-2}
 A_1 = - \langle d \  \cdot  \ , \ovl{b_1} \rangle_x
 + \left(\ovl{c}  -\d_t  \ovl{b_0}  - \div( \ovl{b_1})  \right) ,
 \quad  \text{with } a_1(t,x,\xi) =- i \langle \xi , \ovl{b_1}(t,x)  \rangle_x,
 \enan
and all (high-frequency) results proved for~\eqref{e:wave-HF-obs} yield a counterpart for~\eqref{e:KG-obs}-\eqref{e:def-op-elliptiq-adj}.

\medskip
We now focus on equation~\eqref{e:wave-HF-obs}. 
For $(v_0, v_1) \in H^s \times H^{s-1}$, we recall that there exists a unique solution $v \in \Con^0(0,T_0; H^s(M)) \cap \Con^1(0,T_0; H^{s-1}(M))$ to~\eqref{e:wave-HF-obs}. We set
\begin{align}
\label{eq: splitting time}
v^+(t) =  \big( D_t  + \Lambda \big) v(t)   , \qquad  v^-(t) =  \big( D_t  - \Lambda \big) v(t)  
\end{align}
so that $v^\pm \in \Con^0(0,T_0; H^{s-1}(M))$ for $(v_0, v_1) \in H^s \times H^{s-1}$. We have
\bnan
\label{eq: splitting time inverse}
v(t) = \frac12 \Lambda^{-1} \big( v^+(t) - v^-(t) \big), \qquad D_t v(t) = \frac12  \big( v^+(t) + v^-(t) \big).
\enan
This corresponds to the splitting $(v^+ , v^-) = \widetilde{\Sigma}(v, \d_t v)$ with 
\begin{equation}
\label{e:splitting-sigma-tilde}
\widetilde{\Sigma}  =  \left(
\begin{array}{cc}
\Lambda  &1/ i  \\
- \Lambda & 1/i 
\end{array}
\right) , \qquad 
\widetilde{\Sigma}^{-1}  =  \frac12 
\left(
\begin{array}{cc}
\Lambda^{-1}  & -\Lambda^{-1}  \\
i  &  i
\end{array}
\right) .
\end{equation}
Note that this is not exactly the splitting $\Sigma$ introduced in~\eqref{e:splitting-sigma} but we have
$$
\Sigma  =\frac12 \Lambda^{-1} \left(
\begin{array}{cc}
1  & 0  \\
0  &  -1
\end{array}
\right)
\widetilde{\Sigma} .
$$
We could also have performed the analysis in Section~\ref{s:model-case} with $\widetilde{\Sigma}$, but in the case of the Klein Gordon equation, $\Sigma$ was more convenient to work with in $H^s \times H^s$.

\medskip
Then, writing $\d_t^2 -\Delta +1 = - \big( D_t  + \Lambda \big)\big( D_t  - \Lambda \big)$, Equation~\eqref{e:wave-HF-obs} can be recast as a system of two first order hyperbolic equation in terms of $v^\pm$, namely
\begin{equation}
\begin{cases}
-(D_t-\Lambda) v^+ + \frac{A_0}{2}(v^++v^-) + \frac{A_1 \Lambda^{-1}}{2}(v^+ - v^-) = 0 \\
-(D_t+\Lambda) v^- + \frac{A_0}{2}(v^++v^-) + \frac{A_1 \Lambda^{-1}}{2}(v^+ - v^-) = 0 .
\end{cases}
\end{equation}
This is a striclty hyperbolic Cauchy problem~ \cite[Chapter~7.7]{Taylor:vol2}  with solution operator $\mathscr{S}(t,s)$. As in the scalar case (see Corollary~\ref{cor:S(t,s)}), it enjoys the regularity  
\bna
\mathscr{S}(t,s) \in 
\Bo( (0,T_0)^2 ;\L(H^\sigma (M; \C^2) )) , \\
\d_t \mathscr{S}(t,s) , \d_s \mathscr{S}(t,s) \in \Bo ( (0,T_0)^2;\L(H^\sigma (M; \C^2); H^{\sigma - 1}(M; \C^2))) , 
\ena
for all $\sigma \in \R$. The definition of the operators $\d_t \mathscr{S}(t,s) , \d_s \mathscr{S}(t,s)$ is given in Corollary~\ref{cor:S(t,s)} (in the scalar case). It can be rewritten as
\begin{equation}
\label{e:syste-ondes-lot}
\begin{cases}
(D_t-\Lambda) v^+ - A_+ v^+ - A_- v^- = 0 , \\
(D_t+\Lambda) v^- - A_+ v^+ - A_- v^- = 0 ,
\end{cases}
\end{equation}
with
$$
A_+ = \frac12 \big(A_0 + A_1 \Lambda^{-1} \big) , \qquad A_- = \frac12 \big(A_0 - A_1 \Lambda^{-1} \big) ,
$$
both belonging to $\Cinf(0,T_0 ; \Psi_{\phg}^0(M))$. 
Note that the equations are only coupled by zero order terms. 
Again, this is $P V = 0$ with $V = \transp (v^+, v^- )$ and 
\bnan
\label{e:def-M}
P = D_t + M - A , \qquad 
M= 
\left(
\begin{array}{cc}
- \Lambda  & 0 \\
0 & \Lambda 
\end{array}
\right) 
, \qquad A = 
\left(
\begin{array}{cc}
A_+ & A_-\\
A_+ & A_-
\end{array}
\right) .
\enan

With this splitting in hand, we first have the following high-frequency representation formula for solutions of~\eqref{e:wave-HF-obs} or~\eqref{e:syste-ondes-lot}.
\begin{proposition}
\label{prop:representation}
We denote by $S_\pm (t, s)$ the solution operator associated to $(\d_t \pm i \Lambda - iA_\pm)$, that is $y (s') = S_\pm (s', s)y (s)$ if and only if 
\bnan
\label{e:decoupled-eqn-n}
(\d_t \pm i \Lambda - iA_\pm(t))y (t) = 0 , \quad \text{for all } t\in [s,s'] .
\enan
We also define
\bnan
\label{e:def-S+-}
\mathcal{S}(t,s) =  \left(
\begin{array}{cc}
S_+(t,s) & 0 \\
0    &  S_-(t,s)
\end{array}
\right) .
\enan
Then the solution operator $\mathscr{S}(t,s)$ of~\eqref{e:syste-ondes-lot} satisfies 
$$
\mathscr{S}(t,s) = \mathcal{S}(t,s) + \mathcal{R}(t,s) , \qquad (t,s) \in [0,T_0]^2 ,
$$
where, for all $\sigma \in \R$,
\bnan
\label{e:regularity-R1}
\mathcal{R}(t,s)  \in  \Bo( (0,T_0)^2 ;\L(H^\sigma (M; \C^2); H^{\sigma+1} (M; \C^2) )) ,\\
\label{e:regularity-R2}
\d_t\mathcal{R}(t,s) , \d_s \mathcal{R}(t,s)  \in  \Bo( (0,T_0)^2;\L(H^\sigma (M; \C^2) )) .
\enan

\end{proposition}

\bnp[Proof of Proposition~\ref{prop:representation}]
We use a trick (due to Taylor~\cite[Section~2]{Taylor:75}) to decouple the equations. More precisely, we look for $K \in \Cinf(0,T_0 ; \Psi_{\phg}^{-1}(M;\C^2))$ so that the function $W=(\id - K)V$ solves a diagonal system, up to appropriate remainders (on the variable $V$). We have on the one hand 
\bna
(\id + K)W = V - K^2 V,
\ena
 and hence
\bnan
\label{e:decoupling1}
(\id - K)P(\id + K)W = (\id - K)P ( V - K^2 V) = -  (\id - K)P  K^2 V = R V ,
\enan
since $P V = 0$. Moreover, the remainder satisfies
$ R \in \mathscr{R}^{-1}$, where
$$
  \mathscr{R}^{-1} =  \Cinf(0,T_0 ; \Psi_{\phg}^{-1}(M;\C^2)) + \Cinf(0,T_0 ; \Psi_{\phg}^{-2}(M;\C^2)) D_t  $$
  is the admissible class of remainders in the present context.
On the other hand, we have
\bnan
\label{e:decoupling2}
(\id - K)P(\id + K)W = PW + [P,K]W - KPK W ,
\enan
with $KPK \in  \mathscr{R}^{-1}$. We then remark that $[D_t , K] W = (D_t K) W$ so that $[D_t , K] \in  \Cinf(0,T_0 ; \Psi_{\phg}^{-1}(M;\C^2)) \subset \mathscr{R}^{-1}$, and as well $[A, K] \in \mathscr{R}^{-1}$. Hence, if we can find $K$ such that 
\bnan
\label{e:decouplingK}
- \left(
\begin{array}{cc}
0 & A_-\\
A_+ & 0
\end{array}
\right)  + 
[M , K]\in \mathscr{R}^{-1} ,
\enan
we will then obtain from~\eqref{e:decoupling1}-\eqref{e:decoupling2} that $W$ solves
\bnan
\label{e:eq-W}
P_d W= R_1 W + R_2 V =RV ,
\enan
with $R_1 ,R_2 , R \in \mathscr{R}^{-1}$ and, with $M$ defined in~\eqref{e:def-M}, 
$$
P_d = D_t + M - A_d 
, \qquad A_d = 
\left(
\begin{array}{cc}
A_+ & 0\\
0     & A_-
\end{array}
\right) .
$$
Now taking (for instance)
$$
K := \frac{1}{2} \left(
\begin{array}{cc}
0 &  - \Lambda^{-1}A_+\\
  A_- \Lambda^{-1}    & 0
\end{array}
\right)
\in 
\Cinf(0,T_0 ; \Psi_{\phg}^{-1}(M;\C^2))
$$
realizes~\eqref{e:decouplingK}, and we are left to study $P_d W= RV$, $R \in \mathscr{R}^{-1}$, with $W=(\id - K)V$.

With $\mathcal{S}(t,s)$ defined in~\eqref{e:def-S+-},  Equation~\eqref{e:eq-W} is now solved by 
$$
W(t) = \mathcal{S}(t,s)W(s) + \int_s^t \mathcal{S}(t,t') R(t') V(t') dt' , \qquad R \in \mathscr{R}^{-1} .
$$
Recalling that $W=(\id - K)V$ and that $V(t) = \mathscr{S}(t,s)V(s)$, this yields
$$
V(t) = \mathcal{S}(t,s)V(s) + K(t)\mathscr{S}(t,s)V(s) - \mathcal{S}(t,s)K(s)V(s)+\Big( \int_s^t \mathcal{S}(t,t') R(t') \mathscr{S}(t',s) dt' \Big) V(s) .
$$
This can be rewritten as
$$
V(t) = \mathcal{S}(t,s)V(s) + \mathcal{R}(t,s) V(s), 
$$
with 
$$
\mathcal{R}(t,s) = K(t)\mathscr{S}(t,s)  - \mathcal{S}(t,s)K(s) +\Big( \int_s^t \mathcal{S}(t,t') R(t') \mathscr{S}(t',s) dt' \Big) $$
satisfying 
\bna
\mathcal{R}(t,s)  \in \Bo( (0,T_0)^2 ;\L(H^\sigma (M; \C^2); H^{\sigma+1} (M; \C^2) )), \\
\d_t \mathcal{R}(t,s), \d_s \mathcal{R}(t,s)  \in  \Bo( (0,T_0)^2;\L(H^\sigma (M; \C^2) )) , 
\ena
 for all $\sigma \in \R$, according to the respective regularity properties of $\mathscr{S}(t,s),  \mathcal{S}(t,s)$ and $K(s)$ (see Appendix~\ref{s:pseudo-elementaire} for the regularity properties of $\mathcal{S}(t,s),\mathscr{S}(t,s)$).
\enp
\begin{remark}
Note that the decoupling of the two equations is permitted since the difference of the two eigenvalues of the principal part of the system, namely $\pm \lambda$, is elliptic. Moreover, we do no have the choice of the principal symbol of $K$ in this procedure.
Also, we could choose $K$ by a classical iterative procedure so that all remainders are infinitely smoothing, which is not needed here.
\end{remark}

\begin{remark}
\label{rem:Lambda}
Note here that we do not need to use that $\Lambda$ (the square root of the Laplace operator defined via spectral theory) is a pseudodifferential operator. Indeed, we could in place of $\Lambda$ use any operator $P$ such that
\begin{itemize}
\item $P \in \Psi^1_{\phg}(M)$ with $\sigma_1(P)(x,\xi) = \lambda(x,\xi)=|\xi|_x$
\item $P$ is selfadjoint on $L^2(M)$,
\item $P$ is positive, in the sense that $(P u, u)_{L^2(M)} \geq C \|u\|_{L^2(M)}^2$,
\end{itemize}
Then notice that we have  $-\Delta - P^2 \in \Psi^1_{\phg}(M)$, with principal symbol $\sigma_1(-\Delta - P^2)$ real since $-\Delta - P^2$ is selfadjoint on $L^2(M)$. As a consequence writing Equation~\eqref{e:wave-HF-obs} with $P^2$ instead of $\Lambda^2 = -\Delta + 1$ only amounts to add to $A_1$ a term with real principal symbol. Then, we conclude by remarking that the result of Proposition~\ref{th: HUM operator-waves-total} only depends on $\Im(a_1)$.

Such an operator $P$ is easy to construct using only basic pseudodifferential calculus on $M$: Start with some $A \in \Psi^1_{\phg}(M)$ with $\sigma_1(A)(x,\xi) = \lambda(x,\xi)$ (given by any quantification of the symbol $\lambda$), and set $P:= \frac12(A+A^*) + C_0$ with $C_0$ large enough so that $P$ is positive (use for that the G{\aa}rding inequality). Then it is clear that $P$ fulfills all above conditions.

In Section~\ref{s:model-case}, it was convenient to take an exact square root $\Lambda$, so that to have the nice exact formula~\eqref{eq: HUM Duhamel splitting 1}. The analysis below shows this is not needed.
\end{remark}

The representation formula of Proposition~\ref{prop:representation} together with an appropriate Egorov theorem (Theorem~\ref{theorem: Egorov}) allows to express the Gramian control operator as follows.

\begin{proposition}
\label{th: HUM operator-waves-total}
Denoting by $V_0 = (v_0,  v_1) \in H^s(M) \times H^{s-1}(M)$ the initial data for System~\eqref{e:wave-HF-obs}, and $\widetilde{\Sigma} V_0 = \transp\left(\frac{v_1}{i} + \Lambda v_0 , \frac{v_1}{i} - \Lambda v_0  \right)$, we have 
\begin{equation}
\label{eq: obs = G_T complete}
\int_0^T \|\obs v(t)\|_{H^s(M)}^2 dt  = \big(\G_T \widetilde{\Sigma} V_0, \widetilde{\Sigma} V_0 \big)_{H^{s-1}(M) \times H^{s-1}(M)} ,
\end{equation}
where $\G_T = G_T + R_T$ with $R_T   \in \Bo \big(0,T_0; \L(H^\sigma(M), H^{\sigma +1 }(M; \C^2))\big)$ for all $\sigma \in \R$, and $G_T  \in \Cinf (0,T_0 ;\Psi_{\phg}^0(M ;\C^{2\times2}))$ has principal symbol
\bna
\sigma_0(G_T)   := \frac14 
\left(
\begin{array}{cc}
g_T^+& 0 \\
0  &  g_T^-
\end{array}
\right)
\in S_{\phg}^{0}(T^*M , \C^{2 \times 2}) ,
\ena
with 
$$
g_T^\pm (\rho)=  \int_0^T  \obs^2 \circ \pi \circ \varphi^\pm_{t}(\rho)  e^{-\int_0^t \Im(a_0 \pm a_1 \lambda^{-1}) (\tau, \varphi^\pm_{\tau}(\rho) ) d\tau } dt . 
$$
\end{proposition}

\begin{remark}
Similarly, we also recover an analogue of \cite[Lemma~3.1]{Leb:96} which is the crucial step towards the estimate of the optimal exponential decay rate for the damped wave equation. Namely, for all $T>0$, there is a constant $C>0$ such that we have, for all solutions of $\d_t^2 v -\Delta v + b_0 \d_t v  = 0$, 
\bna
\E_1(v, \d_tv)(T) \leq \exp \left(-2 \inf_{(x, \xi) \in S^*M}\int_0^T b_0(s, x(s)) ds\right)\E_1(v, \d_tv)(0)  + C \E_0(v, \d_tv)(0) , \\
\E_1(v, \d_tv)(T) \geq \exp \left(-2 \sup_{(x, \xi) \in S^*M}\int_0^T b_0(s, x(s)) ds\right)\E_1(v, \d_tv)(0)  - C \E_0(v, \d_tv)(0) ,
\ena
where $x(s) = \pi \circ \varphi_s(x, \xi)$. The proof is very close to that of Proposition~\ref{th: HUM operator-waves-total}: it follows from the representation formula of Proposition~\ref{prop:representation}, the Egorov Theorem~\ref{theorem: Egorov}, and the sharp G{\aa}rding estimate.
\end{remark}
\bnp[Proof of Proposition~\ref{th: HUM operator-waves-total}]
According to~\eqref{eq: splitting time inverse}, the unique solution to~\eqref{e:wave-HF-obs} is given by 
$$v(t) = \frac12 \Lambda^{-1} \big( v^+(t) - v^-(t) \big) = L V(t),
$$
where
$$V(t) = \transp( v^+(t) , v^-(t)) , \quad \text{ and }\quad L: = \frac12 \Lambda^{-1} (1, -1) .
$$ 
According to Proposition~\ref{prop:representation}, $V(t) = \transp(v^+(t) , v^-(t))$ satisfies
$$
V(t) = \mathscr{S}(t,0)V^0 , \qquad \mathscr{S}(t,0) = \mathcal{S}(t,0)  + \mathcal{R}(t,0)  , \qquad t  \in [0,T_0]  ,
$$
with  
$$
V^0 = (v^+_0 , v^-_0 ) =\widetilde{\Sigma}(v_0,v_1) = \left(\frac{v_1}{i} + \Lambda v_0 , \frac{v_1}{i} - \Lambda v_0  \right) \in H^{s-1}(M ; \C^2) ,
$$
for $(v_0,  v_1) \in H^s(M) \times H^{s-1}(M)$.
Now, we compute 
\bna
\int_0^T \|\obs v(t)\|_{H^s(M)}^2 dt & =&  \int_0^T \|\obs L \mathscr{S}(t,0)V^0 \|_{H^s(M)}^2 dt  \\
 & =&  \int_0^T \left( \mathscr{S}(t,0)^* \transp L \obs \Lambda^{2s}\obs L \mathscr{S}(t,0)V^0, V^0 \right)_{L^2(M; \C^2)} dt  \\
\ena
where all adjoints are taken in $L^2$.
This implies
\bna
\int_0^T \|\obs v(t)\|_{H^s(M)}^2 dt =  \left( \mathcal{G}_T V^0, V^0 \right)_{H^{s -1}(M; \C^2)}   ,
 \quad \mathcal{G}_T = \int_0^T  \Lambda^{2(1-s)}\mathscr{S}(t,0)^* \transp L \obs \Lambda^{2s}\obs L \mathscr{S}(t,0)dt .
\ena
Recalling now the form of $\mathscr{S}(t,0) = \mathcal{S}(t,0) + \mathcal{R}(t,0)$ given by Proposition~\ref{prop:representation}, we set 
\bnan
\label{e:def Gtilde}
\tilde{\mathcal{G}}_T   := 
 \int_0^T  \Lambda^{2(1-s)}\mathcal{S}(t,0)^* \transp L \obs \Lambda^{2s}\obs L \mathcal{S}(t,0)dt ,
  \quad \text{and} \quad \tilde{R}_T = \mathcal{G}_T - \tilde{\mathcal{G}}_T
\enan
The regularity properties of  $\Lambda^{2(1-s)}$, $\mathscr{S}(t,0)$, $L$, and that of $\mathcal{R}(t,0)$ given in~\eqref{e:regularity-R1}-\eqref{e:regularity-R2} yield that 
\bna
\tilde{R}_T \in \Bo( (0,T_0) ;\L(H^\sigma (M; \C^2); H^{\sigma+1} (M; \C^2) )) .
\ena
Next, recalling the definition of $\mathcal{S}(t,0)$ in~\eqref{e:def-S+-}, we can compute
\bnan
\label{e:evolgram}
\mathcal{S}(t,0)^* \transp L \obs \Lambda^{2s}\obs L \mathcal{S}(t,0) = 
\left(
\begin{array}{cc}
S_+(t,0)^*  B S_+(t,0) & - S_+(t,0)^*  B S_-(t,0)\\
- S_-(t,0)^*  B S_+(t,0) & S_-(t,0)^*  B S_-(t,0) 
\end{array}
\right) , 
\enan
with
$$
B = \frac14 \Lambda^{-1} \obs \Lambda^{2s} \obs \Lambda^{-1} \in  \Psi^{2s-2}_{\phg}(M).
$$
Let us first study the diagonal terms in~\eqref{e:evolgram}. With $S_\pm(t,0)$ defined in~\eqref{e:decoupled-eqn-n}, the Egorov theorem~\ref{theorem: Egorov} yields the existence of $Q_\pm (t ) \in \Cinf  \big((0,T_0), \Psi^{2s-2}_{\phg}(M)\big)$ 
 and
 \bna
 R_\pm (t ) \in \Bo \big((0,T_0) , \L(H^\sigma(M), H^{\sigma +1-2(s-1)}(M))\big) ,\\
 \d_t R_\pm (t ) \in \Bo \big((0,T_0), \L(H^\sigma(M), H^{\sigma -2(s-1)}(M))\big),
 \ena
for all $\sigma \in \R$, such that we have 
 $$
 S_\pm (t,0)^*  B S_\pm (t,0)  - Q_\pm(t ) = R_\pm(t ) , \qquad  t  \in (0,T_0) .
 $$ 
and the principal symbol of $Q_\pm(t )$ is given by 
 $$
 q_\pm (t , \rho)= \frac14 \lambda^{2s-2} \obs^2 \circ \pi \circ \varphi^{\pm}_{t}(\rho)  e^{2\int_t^0 \Im(a_\pm) (\tau, \varphi^\pm_{\tau}(\rho) ) d\tau }\quad \in
 \Cinf((0,T_0) , S^{2s-2}_{\phg}(T^\ast M)) ,
 $$
 where $a_{\pm} = \sigma_0(A_\pm)$.
 
Concerning the anti-diagonal terms in~\eqref{e:evolgram} when integrated on $(0,T)$, Lemma~\ref{lemma: regularity Hs non self} yields 
$$
\int_0^T S_\pm(t,0)^*  B S_\mp(t,0) dt \in \Bo \big((0,T_0) , \L(H^\sigma(M), H^{\sigma +1-2(s-1)}(M))\big) .
$$
 
With all these properties in hand, when coming back to~\eqref{e:def Gtilde}, we may now write $\tilde{\mathcal{G}}_T:= G_T +R^0_T$
where $R^0_T   \in \Bo \big((0,T_0), \L(H^\sigma(M; \C^2), H^{\sigma +1 }(M; \C^2))\big)$ for all $\sigma \in \R$, and $G_T$ is given by 
\bna
G_T   := 
\left(
\begin{array}{cc}
\Lambda^{2(1-s)} \int_0^T Q_+(t ) dt& 0 \\
0  & \Lambda^{2(1-s)} \int_0^T Q_-(t )  dt
\end{array}
\right) 
 \quad \in \Cinf  \big(0,T_0;\Psi^{0}_{\phg}(M)\big) ,
\ena
and has principal symbol
\bna
\sigma_0(G_T)   := \frac14 
\left(
\begin{array}{cc}
 \int_0^T  \obs^2 \circ \pi \circ \varphi^+_{t}(\rho)  e^{-2\int_0^t \Im(a_+) (\tau, \varphi^+_{\tau}(\rho) ) d\tau } dt& 0 \\
0  &  \int_0^T \obs^2 \circ \pi \circ \varphi^-_{t}(\rho)  e^{-2\int_0^t \Im(a_-) (\tau, \varphi^-_{\tau}(\rho) ) d\tau } dt
\end{array}
\right) .
\ena

This, together with~\eqref{e:def Gtilde} concludes the proof of the Proposition.
\enp

As a consequence of Proposition~\ref{th: HUM operator-waves-total}, we obtain the following high-frequency observability estimate. We use for this the definition of the constant $\mathfrak{K}(T)$ associated to~\eqref{e:wave-HF-obs}:
\bnan
\label{e:KTgeneral}
\mathfrak{K}(T)  := \min \left\{ \min_{\rho \in S^*M} g_T^+(\rho) , \min_{\rho \in S^*M} g_T^-(\rho)\right\} .
\enan

\begin{proposition}
\label{l:obsHF-general}
For any $T_0>0$, there exists a constant $C_0>0$ such that for all $T \in [0,T_0]$, for all $V_0 =(v_0, v_1) \in H^s(M)\times H^{s-1}(M)$ and associated solution $v$ of~\eqref{e:wave-HF-obs}, we have
\begin{equation}
\label{eq: obs HF-general}
\int_0^T \|\obs v (t)\|_{H^s(M)}^2 dt  \geq \mathfrak{K}(T)  \E_s(V_0) -C_0 \E_{s-1/2}(V_0),
\end{equation}
where $\mathfrak{K}(T)$ is defined by~\eqref{e:KTgeneral}.
\end{proposition}

Note that in the case of Equation~\eqref{e:KG-obs}-\eqref{e:def-op-elliptiq-adj} above, the symbols $a_0 ,a_1$ are given by \eqref{e:lot-link-a-b-1}-\eqref{e:lot-link-a-b-2}, so that in this case, denoting by $(x^\pm(s) , \xi^\pm(s)) = \varphi_s^\pm (x_0 ,\xi_0)$, we have
\bnan
\label{e:def-gpm-pde}
g_T^\pm (x_0 , \xi_0)=  \int_0^T  \obs^2(x^\pm(t))  \exp \left(\int_0^t \Re(b_0)(\tau, x^\pm(\tau)) \pm \left< \frac{\xi^\pm(\tau)}{|\xi^\pm (\tau)|_{x^\pm(\tau)}} , \Re(b_1)(\tau , x^\pm(\tau)) \right>_{x^\pm (\tau)}  d\tau \right) dt.
\enan
The two functions $g_T^-$ and $g_T^+$ in~\eqref{e:def-gpm-pde} are linked by the following lemma, proved in Appendix~\ref{app:geom}.
\begin{lemma}
\label{l:egal-gpm-pde}
With $g_T^-$ and $g_T^+$ given by~\eqref{e:def-gpm-pde} we have $g_T^-  \circ \sigma = g_T^+$, where $\sigma(x,\xi) = (x, -\xi)$.
\end{lemma}
According to Lemma~\ref{l:egal-gpm-pde} (together with the fact that $\sigma$ is an involution), we have in this situation  $\min_{\rho \in S^*M} g_T^+(\rho) = \min_{\rho \in S^*M} g_T^-(\rho)$. This justifies the definition~\eqref{e:KTgeneral+} in Theorem~\ref{t:main-result-lot}.

\bnp[Proof of Proposition~\ref{l:obsHF-general}]
We follow the proof of Proposition~\ref{l:obsHF}.
From Proposition~\ref{th: HUM operator-waves-total} and the use of the uniform G{\aa}rding estimate of Theorem~\ref{th: sharp garding manifold} (or its corollary), we obtain, uniformly for $T \in [0,T_0]$,
\bna
\int_0^T \|\obs v(t)\|_{H^s(M)}^2 dt  & = & \big(\G_T \widetilde{\Sigma} V_0, \widetilde{\Sigma} V_0 \big)_{H^{s-1}(M) \times H^{s-1}(M)} \\
& \geq & \frac14 \mathfrak{K}(T) \| \widetilde{\Sigma} V_0\|_{H^{s-1}(M; \C^2)}^2 - C_0 \| \widetilde{\Sigma} V_0\|_{H^{s-3/2}(M; \C^2)}^2 .
\ena
To conclude, we just notice that  
\bna
\| \widetilde{\Sigma} V_0 \|_{H^{s-1}(M; \C^2)}^2 & = & \| \frac{v_1}{i} + \Lambda v_0\|_{H^{s-1}(M)}^2 + \| \frac{v_1}{i} - \Lambda v_0\|_{H^{s-1}(M)}^2\\
& = & 2 \| v_1 \|_{H^{s-1}(M)}^2 + 2 \| v_0 \|_{H^{s}(M)}^2 = 4 \mathscr{E}_s(V_0) .
\ena
\enp

\section{Uniform dependence with respect to potentials}
\label{s:potential}

In this section, we allow $M$ to have a nonempty boundary $\d M$. In fact, we do not perform a high-frequency analysis but rather use as a black box a known result, for which we refer \eg to~\cite{BLR:92},~\cite{Leb:96}.
We hence now use the notation: $\H^1 = H^1_0(M) \times L^2(M)$, $\H^1 =  L^2(M) \times H^{-1}(M)$ ($H^{-1}$ being the usual dual space of $H^1_0$), and 
$$
\E_1(u , \d_t u) = \frac12 \left( \|\d_t u \|_{L^2(M)}^2 +\|\nabla u \|_{L^2(M)}^2+\| u \|_{L^2(M)}^2\right). 
$$

In Section~\ref{s:potential-lowfreq}, we first focus on obaining (from \cite{LL:15}) an explicit dependence of the low frequency estimates with respect to potentials. We then conclude the proof in Section~\ref{s:potential-fullest}.

\subsection{The low-frequency estimate}
\label{s:potential-lowfreq}
Our starting point is the following result, which is a particular case of~\cite[Theorem~6.3]{LL:15}, when there is no first order terms.
\begin{theorem}
\label{thmobserwave-bis}
For any nonempty open subset $\omega$ of $M$ and any $T> \mathcal{L}(M ,\omega)$, there exist $\eps, C, \kappa ,\mu_0>0$ such that for any $c\in L^{\infty}(M)$, any $u \in H^1((-T,T) \times M)$ solving \eqref{e:waves-b-eq-c},
we have, for any $\mu\geq \mu_0\max\{1 , \|c\|_{L^\infty}^{\frac23}\}$,
\bna
\nor{u}{L^2((-\eps ,\eps) \times M)}\leq C e^{\kappa \mu}\nor{u}{L^2((-T,T); H^1(\omega))} 
+\frac{C}{\mu}\nor{u}{H^1((-T,T) \times M)}.
\ena
If $\partial M\neq \emptyset$ and $\Gamma$ is a non empty open subset of $\partial M$, for any $T>\mathcal{L}(M ,\Gamma)$, there exist $\eps , C, \kappa,\mu_0 >0$ such that for any $u \in H^1((-T,T) \times M)$ solving~\eqref{e:waves-b-eq-c}, we have
\bna
\nor{u}{L^2((-\eps ,\eps) \times M)}\leq C e^{\kappa \mu}
\nor{\partial_{\nu}u}{L^2((-T,T)\times \Gamma)} 
+\frac{C}{\mu}\nor{u}{H^1((-T,T) \times M)}.
\ena
\end{theorem}
From this result, we may deduce, in case there is no first order terms, the following corollary which is a refined version of~\cite[Theorem~6.1]{LL:15} (in which we replace $C=C_0 e^{C_0 \|c\|_{L^\infty}}$ by $C=C_0 e^{C_0\sqrt{\|c\|_{L^\infty}}}$)
\begin{corollary}
\label{c:corollaire-QCP}
Under the same assumptions as Theorem \ref{thmobserwave-bis}, there exist $C_0, \kappa ,\mu_0>0$ such that for any $c\in L^{\infty}(M)$, any $u \in H^1((-T,T) \times M)$ solving \eqref{e:waves-b-eq-c},
we have,
for any $\mu\geq \mu_0\max\{1 , \|c\|_{L^\infty}^{\frac23}\}$,
\bna
\nor{(u_0,u_1)}{\H^0}\leq C e^{\kappa \mu}\nor{u}{L^2((-T,T); H^1(\omega))} 
+\frac{C}{\mu}\nor{(u_0,u_1)}{\H^1},
\ena
resp., in the boundary observation case,
\bna
\nor{(u_0,u_1)}{\H^0}\leq C e^{\kappa \mu}\nor{\partial_{\nu}u}{L^2((-T,T)\times \Gamma)} 
+\frac{C}{\mu}\nor{(u_0,u_1)}{\H^1},
\ena
with $C=C_0 e^{C_0\sqrt{\|c\|_{L^\infty}}}$.
\end{corollary}
These estimates will eventually lead to the general bound of the form $\mathfrak{C}_{obs}= C\exp (\exp(C\nor{c}{L^{\infty}(M)}^{1/2}))$.
This result is a direct consequence of the following lemma of energy estimates.
\begin{lemma}
There exists $C>0$ such for any $u$ solution of \eqref{e:waves-b-eq-c}, we have
\bnan
\label{NRJsqrtH1}
\nor{(u(t),\d_t u(t))}{\H^1}\leq C e^{C|t-s|\sqrt{\|c\|_{L^\infty}}}\nor{(u(s),\d_t u(s))}{\H^1} ,\\
\label{NRJsqrtL2} \nor{(u(t),\d_t u(t))}{\H^0}\leq C e^{C|t-s|\sqrt{\|c\|_{L^\infty}}}\nor{(u(s),\d_t u(s))}{\H^0}.
\enan
For any $T>0$ there exists $C>0$ such that  for any $u$ solution of \eqref{e:waves-b-eq-c}, we have
\bnan
\label{NRJintH1}C^{-1} e^{-C\sqrt{\|c\|_{L^\infty}}} \nor{u}{H^1((-T,T) \times M)}\leq &\nor{(u_0,u_1)}{\H^1}&\leq C e^{C\sqrt{\|c\|_{L^\infty}}} \nor{u}{H^1((-T,T) \times M)} ,\\
\label{NRJintL2}C^{-1} e^{-C\sqrt{\|c\|_{L^\infty}}}\nor{u}{L^2((-T,T) \times M)}\leq &\nor{(u_0,u_1)}{\H^0}&\leq C e^{C\sqrt{\|c\|_{L^\infty}}}\nor{u}{L^2((-T,T) \times M)}.
\enan
\end{lemma}
The nontrivial part of this Lemma is in the power $1/2$ for the size of the potential. Estimates~\eqref{NRJsqrtH1} and \eqref{NRJsqrtL2} are proved in \cite{DZZ:08} using a modified energy method (see estimate (2.50) and (2.44) in that reference, see also~\cite{Zua:93}). Both estimates in~\eqref{NRJintH1} and the first part of \eqref{NRJintL2} are obtained by integration on $(-T,T)$. The second estimate of \eqref{NRJintL2} is obtained from \eqref{NRJintH1} by a duality argument (see the proof of Theorem 6.1 in \cite{LL:15}). A similar argument will be performed in the proof of Lemma \ref{lminegTE}.

\bigskip
In the case when $c$ belongs to $L^\infty_{\delta}$, the exponential dependence with respect to $c$ in the constant $\mathfrak{C}_{obs}$ can in fact be improved.
We stress the fact that potentials in $L^\infty_{\delta}$ are real-valued so that $-\Delta + c$ is selfadjoint on $L^2$. 
If  $c \in L^\infty_0$, the operator $-\Delta + c$ is nonnegative.

Setting
$$
\E_c(u , \d_t u) = \frac12 \left( \|\d_t u \|_{L^2(M)}^2 +\|\nabla u \|_{L^2(M)}^2 + \int_M c|u|^2\right) ,
$$ 
we always have
\bna
\E_c(u , \d_t u)\leq (1+\|c\|_{L^\infty(M)} )\E_1(u, \d_tu )   ,
\ena
and, if $c \in L^\infty_\delta$ with $\delta>0$ we also obtain
\bnan
\label{E_1-E_c}
 \E_1(u , \d_t u) & = & \frac12 \left( \|\d_t u \|_{L^2(M)}^2 +\| \nabla u \|_{L^2(M)}^2 + \|  u \|_{L^2(M)}^2 \right) \nonumber \\
&\leq &\frac12 \left( \|\d_t u \|_{L^2(M)}^2+ \| \nabla u \|_{L^2(M)}^2  +\delta^{-1}\left(\int_M |\nabla u|^2 + c|u|^2  \right)\right) \nonumber\\
&\leq &(1+ \delta^{-1})  \E_c(u, \d_tu ) . 
\enan
We have the following elementary Lemma which applies for any $c \in L^\infty_{\delta}\subset L^\infty_0$, $\delta\geq 0$.
\begin{lemma}
\label{lmNRJ}
Let $T>0$. Then, there exists $C_T>0$ such that for all $c \in L^\infty_0$, all $(u_0,u_1) \in H^1_0(M) \times L^2(M)$, $g \in L^1(0,T ; L^2(M))$ and $u$ associated solution of 
\bneqn
\label{wave-c-g}
\partial_t^2u-\Delta u+cu&=&g ,\\
u_{|\d M} & = & 0 , \quad  \text{if } \d M \neq \emptyset, \\
(u,\partial_t u)_{t=0}&=&(u_0,u_1) ,
\eneqn
we have the estimate
\bna
\sup_{t\in [0,T]}(\E_c(u , \d_t u))\leq C_T \left(\E_c(u_0 , u_1) + \nor{g}{L^1(0,T ;L^2(M)}^2 \right).
\ena
If moreover $g=0$, then we have $\E_c(u , \d_t u) = \E_c(u_0 , u_1) $ on $(0,T)$.
\end{lemma}
\bnp
Note first that $c \in L^\infty_0$ ensures that $\E_c$ is nonnegative.
Multiply the equation by $\partial_t \overline{u}$, take real part and integrate on $M$ to obtain (at least for smooth solutions)
\bna
\frac{d}{dt} \left(\E_c(u , \d_t u)\right) & = & \int_{M}g(t,x)\partial_t \overline{u}(t,x) \leq \|g(t)\|_{L^2(M)}\|\d_t u(t)\|_{L^2(M)}\\
& \leq &\|g(t)\|_{L^2(M)}  \sqrt{2 \E_c(u , \d_t u)} .
\ena
An appropriate Gronwall inequality gives the expected estimate. The case $g=0$ comes from the first identity.
\enp

Now, we prove the following bound by duality.
\begin{lemma}
\label{lminegTE}
For all $T, \eps >0$ there is $ C_{\eps,T} >0$ such that for all $\delta>0$, $c \in L^\infty_\delta$, all $(u_0,u_1) \in H^1_0(M) \times L^2(M)$, and all associated solution $u \in C^0(0,T; H^1_0(M))\cap C^1(0,T;L^2(M))$ of~\eqref{wave-c-g} with $g=0$, we have, 
\bna
\nor{u}{L^2((-T ,T) \times M)} \leq (1+\delta^{-1})^\frac12 C_{\eps,T}\nor{u}{L^2((-\eps ,\eps) \times M)} .
\ena
\end{lemma}
\bnp
Define $v$ to be the unique (backward) solution to 
\bneq
(\d_t^2 -\Delta +  c) v&=&u\\
v_{\left|\partial  M\right.}&=&0  \\
(v,\partial_t v)_{\left|t=T\right.} &=&(0,0).
\eneq
By integration by parts, we have
\bnan
\label{dual-u-v}
\int_0^T \int_{M}|u|^2=\int_0^T \int_{M}u (\d_t^2 -\Delta +  c) \ovl{v}=\int_{ M}\partial_t u(0)\ovl{v}(0)-\int_{ M}u(0) \d_t \ovl{v}(0). 
\enan
But now, take $\chi\in C^{\infty}([0,T])$ with $\chi=1$ close to $0$ and $\chi=0$ for $t\in [\eps,T]$. Define $w=\chi(t)v$ solution of 
\bneq
(\d_t^2 -\Delta +  c) w&=&\chi u+2\dot{\chi}(t)\partial_t v+ \ddot{\chi}(t)v=:g=: g_1+g_2\\
w_{\left|\partial  M\right.}&=&0  \\
(w,\partial_t w)_{|t=0} &=&(v,\partial_t v)_{|t=0} ,\\
(w,\partial_t w)_{\left|t=T\right.} &=&(0,0) .\\
\eneq
with $g_1=\chi u$.
We have the estimate
\bna
\nor{g_2}{L^2((0,T)\times M)}^2 \leq C \nor{(v,\partial_t v)}{L^2([0,T],\H^1(M))}^2 = 2C \int_0^T \E_1(v,\d_t v) dt.
\ena 
Moreover,~\eqref{E_1-E_c} then yields $\E_1(v,\d_t v) \leq C(1+\delta^{-1}) \E_c(v,\d_t v)$ so that 
\bna
\nor{g_2}{L^2((0,T)\times M)}^2 \leq  C \int_0^T \E_c(v,\d_t v) dt \leq CT(1+\delta^{-1})  \sup_{t\in[0,T]} \E_c(v,\d_t v)(t) .
\ena
Then, the equation satisfied by $v$ together with Lemma \ref{lmNRJ} give
 \bna
\nor{g_2}{L^2((0,T)\times M)}^2 \leq CT(1+\delta^{-1})  \sup_{t\in[0,T]} \E_c(v,\d_t v)(t)  \leq C_T(1+\delta^{-1}) \nor{u}{L^2((0,T)\times M)}^2.
\ena
Since  $g_1=\chi u$ trivially satisfies this estimate, we  finally obtain, with $g=g_1 + g_2$ (we drop the dependence with respect to $T$ or $\eps$) 
 \bnan
 \label{estim-g-delta-T}
\nor{g}{L^2((0,T)\times M)}^2 \leq C(1+\delta^{-1}) \nor{u}{L^2((0,T)\times M)}^2.
\enan
The same computation as in~\eqref{dual-u-v} for $w$, noticing that the boundary value of $w$ are the same as $v$, yields the identity
\bna
\int_0^T \int_{M}u\ovl{g}=\int_0^T \int_{M}u (\d_t^2 -\Delta +  c) \ovl{w}=\int_{ M}\partial_t u(0)\ovl{v}(0)-\int_{ M}u(0) \d_t \ovl{v}(0). 
\ena 
Identifying this right hand-side with that of~\eqref{dual-u-v}, we therefore obtain
\bna
\int_0^T \int_{M}u\ovl{g}=\int_0^T \int_{M}|u|^2. 
\ena 
Moreover, since $g$ is supported in $[0,\eps]$, and using~\eqref{estim-g-delta-T} we have
\bna
\int_0^T \int_{M}|u|^2 & =&  \int_0^{\eps} \int_{M}u\ovl{g}\leq \nor{u}{L^2((0,\eps)\times M)}\nor{g}{L^2((0,\eps)\times M)} \\
& \leq & C (1+\delta^{-1})^\frac12\nor{u}{L^2([0,\eps]\times M)}\nor{u}{L^2((0,T)\times M)}.
\ena
and therefore $\nor{u}{L^2((0,T)\times M)}\leq C (1+\delta^{-1})^\frac12\nor{u}{L^2((0,\eps)\times M)}$. Changing $u(t)$ into $u(-t)$ also leads to $\nor{u}{L^2((-T,0)\times M)}\leq C (1+\delta^{-1})^\frac12\nor{u}{L^2((-\eps,0)\times M)}$, which concludes the proof of the lemma.
\enp
With this refined energy estimates (with respect to those of the proof of~\cite[Theorem~6.2]{LL:15}) and using the quantitative unique continuation result of Theorem~\ref{thmobserwave-bis} above, we can then prove the following result.
\begin{corollary}
\label{thmestimbis}
Let $T>\mathcal{L}(M ,\omega)$ (\resp $T>\mathcal{L}(M ,\Gamma)$). There exist $C, \kappa ,\mu_0>0$ such that for any $(u_0,u_1)\in H^1_0(M)\times L^2(M)$, for any $\delta>0$, any $c\in L^\infty_{\delta}$,
and associated solution $u$ of \eqref{e:waves-b-eq-c}, we have,
with $C_{\delta}  = C (1+ \delta^{-1/2})$,  
the estimate
\bnan
\label{estimthm1-bis}
\nor{u}{L^2((-T ,T) \times M)}
\leq C_{\delta} e^{\kappa \mu} \nor{u}{L^2((-T,T); H^1(\omega_0))} 
+ \frac{C_{\delta} }{\mu}\left(\nor{(u_0,u_1)}{\H^1(M)}+\nor{c u}{L^2((-T ,T) \times M)}\right).
\enan
\resp, the estimate
\bnan
\label{estimthm2-bis}
\nor{u}{L^2((-T ,T) \times M)}\leq 
C_{\delta}  e^{\kappa \mu} \nor{\partial_{\nu}u}{L^2((-T,T)\times \Gamma)}  
 +\frac{C_{\delta}}{\mu}\left(\nor{(u_0,u_1)}{\H^1(M)}+\nor{c u}{L^2((-T ,T) \times M)}\right).
\enan
for all $\mu\geq \mu_0\max\{1 , \|c\|_{L^\infty}^{\frac23} \}$.
\end{corollary}
\bnp
We start again from the above Theorem~\ref{thmobserwave-bis}, namely, for all $\mu\geq \mu_0\max\{1 , \|c\|_{L^\infty}^{\frac23} \}$, we have
 \bna
\nor{u}{L^2((-\eps ,\eps) \times M)}\leq C e^{\kappa \mu} 
\nor{u}{L^2((-T,T); H^1(\omega))}  
+\frac{C}{\mu}\nor{u}{H^1((-T,T) \times M)}.
\ena
Lemma \ref{lminegTE} then gives
 \bna
\nor{u}{L^2((-T ,T) \times M)}\leq C_\delta e^{\kappa \mu} 
\nor{u}{L^2((-T,T); H^1(\omega))}  
+\frac{C_\delta }{\mu}\nor{u}{H^1((-T,T) \times M)}.
\ena
Then, using classical hyperbolic energy estimates, viewing $c u$ as a source term, we have
\bna
\nor{u}{H^1((-T,T) \times M)}\leq C( \nor{(u_0,u_1)}{\H^1(M)}+\nor{c u}{L^2((-T ,T) \times M)}).
\ena
Plugging this last estimate into the previous one yields the sought result.
\enp

\subsection{The full observability estimate}
\label{s:potential-fullest}

We now combine the quantitative unique continuation result of Corollary~\ref{c:corollaire-QCP} (general case) or~\ref{thmestimbis} (case $c \in L^\infty_\delta$) with this result with an observability estimate (or a relaxed observability estimate) for the wave equation without potential (used here as a black box) to prove Theorem~\ref{t:dep-pot}. The following is \eg given in~\cite{BLR:92}.

\begin{theorem}[\!\! \cite{BLR:92}] Assumes that $(\omega, T)$ satisfies GCC, \resp that $(\Gamma, T)$ satisfies GCC$_{\d}$. Then, there exist $C_0,C_1>0$ such that for any $ (w_0, w_1) \in H^1_0(M) \times L^2(M)$, and associated solution $w$ of
 \begin{equation}
\label{e:waves-b-eq}
\begin{cases}
\d_t^2 w-\Delta w = 0 , \\
w|_{\d M} =  0  , \quad \text{if } \d M \neq \emptyset \\
(w(0), \d_t w(0)) = (w_0, w_1) ,
\end{cases} 
\end{equation}
we have
\bnan
\label{e:waves-b-obs}
 \int_0^T \| w(t) \|_{H^1(\omega)}^2 dt \geq C_1 \E_1(w_0,w_1).
\enan
\resp  , 
\bnan
\label{e:waves-bb-obs}
 \int_0^T \| \d_\nu w(t) \|_{L^2(\Gamma)}^2 dt \geq C_1 \E_1(w_0,w_1).
\enan
\end{theorem}
\begin{remark}
Note that we only need \eqref{e:waves-b-obs}-\eqref{e:waves-bb-obs} under the relaxed form~\eqref{e:general-obs-relaxed} (i.e. with a remainder of the form $C \E_0(w_0,w_1)$). Here, it is stated as in~\cite{BLR:92}.
\end{remark}

\bnp[Proof of Theorem~\ref{t:dep-pot}]
Both estimates (boundary and internal observation) are proved the same way, so we only detail e.g. the internal case. We only give details when the proof is different. 

With $w$ solution of~\eqref{e:waves-b-eq} and $v$ solution of~\eqref{e:waves-b-eq-c}, starting from the same initial data $V_0 = (v_0,v_1) = (w_0 ,w_1)$, we have, setting $z = w -v $, 
\begin{equation}
\label{e:waves-b-eq-z}
\begin{cases}
\d_t^2 z -\Delta z = c(x) v, \\
z|_{\d M} =  0  , \quad \text{if } \d M \neq \emptyset \\
 (z, \d_t z)|_{t=0} = (0,0) .
 \end{cases} 
\end{equation}
Then, the hyperbolic energy estimates for $z$ yield 
\bna
\int_{0}^T \| z (t)\|_{H^1(\omega)}^2 dt\leq C\|z\|_{L^\infty(0,T; H^1(M))}^2 
\leq  C \|c \,  v\|_{L^2((0,T); L^2(M))}^2.
\ena
In the case of boundary observation, we will use instead the hidden regularity of the wave equation (see for instance Theorem 4.1 p44 of Lions \cite{Lio:88} in the flat case)
\bna
\int_0^T \| \d_\nu z(t) \|_{L^2(\Gamma)}^2 dt 
\leq  C \|c \,  v\|_{L^1((0,T); L^2(M))}^2.
\ena
Hence, from the observability estimate~\eqref{e:waves-b-obs}, we obtain
\bna
2\int_{0}^T \| z (t)\|_{H^1(\omega)}^2 dt + 2\int_{0}^T \| v (t)\|_{H^1(\omega)}^2 dt \geq \int_{0}^T \| w (t)\|_{H^1(\omega)}^2 dt \geq C_1  \E_1(V_0),
\ena
and hence
\bnan
\label{eq: obs HF-pot}
\int_0^T \|  v (t)\|_{H^1(\omega)}^2 dt  \geq C_0 \E_1(V_0) - C\|c \,  v\|_{L^2((0,T); L^2(M))}^2 .
\enan
Note that we obtain the same estimate for the boundary observation and the reasoning will be exactly the same up to now. So, we only detail the internal case.

Next, in the general case $c \in L^\infty$, we write 
$$
\|c \,  v\|_{L^2((0,T); L^2(M))}^2 \leq \|c\|_{L^\infty}^2 \|v\|_{L^2((0,T); L^2(M))}^2 \leq C \exp(C \|c\|_{L^\infty}^\frac12)\E_0(V_0),
$$
 according to~\eqref{NRJintL2}. 
 Note then that $(\omega, T)$ satisfies GCC implies that $T> T_{UC}(\omega)$ (resp., that $(\Gamma, T)$ satisfies GCC$_\d$ implies that $T> T_{UC}(\Gamma)$) as in the boundaryless case, see Remark~\ref{rem-boundary-T-T}.
 Hence, this estimate, together with~\eqref{eq: obs HF-pot} and Corollary~\ref{c:corollaire-QCP}, yields
\bnan
\label{eq: obs HF-pot-2}
C \exp(C \|c\|_{L^\infty}^\frac12)  e^{\kappa \mu}\int_0^T \|  v (t)\|_{H^1(\omega)}^2 dt 
+\frac{C \exp(C \|c\|_{L^\infty}^\frac12) }{\mu^2}\E_1(V_0) 
+ \int_0^T \|  v (t)\|_{H^1(\omega)}^2 dt  \geq C_0 \E_1(V_0) .
\enan
for any $\mu\geq \mu_0\max\{1 , \|c\|_{L^\infty}^{\frac23}\}$. 
This yields the sought result in the general case $c \in L^\infty$, after having taken $\mu \geq \frac{C}{2C_0} \exp(\frac{C}{2} \|c\|_{L^\infty}^\frac12) $.

\bigskip
From now on, we consider the case $c \in L^\infty_\delta$. The strategy is slightly different. Considering $w$ the solution of~\eqref{e:waves-b-eq} coinciding with $v$ at time $t=T/2$ (instead of $t=0$), we obtain similarly
\bnan
\label{eq: obs HF-pot-bis}
\int_0^T \|  v (t)\|_{H^1(\omega)}^2 dt  \geq C_0 \E_1(v, \d_tv)(T/2) - C\|c \,  v\|_{L^2((0,T); L^2(M))}^2 .
\enan
This uses the observability estimate~\eqref{e:waves-b-obs} together with the fact that $w$ satisfies $\E_1(w, \d_tw)(t) \leq C\E_1(w, \d_tw)(0)$. 

We may now use the quantitative unique continuation result of Corollary~\ref{thmestimbis} to get rid of the term $\|c \,  v\|_{L^2((-T,T); L^2(M))}^2$. Corollary~\ref{thmestimbis} (applied on the time interval $(0,T)$ instead of $(-T,T)$) yields the existence of $C, \kappa ,\mu_0>0$ such that for any $c \in L^\infty(M)$, any $v$ solution of~\eqref{e:waves-b-eq-c}, and any $\mu€ \geq \mu_0\max\{1 , \|c\|_{L^\infty}^{\frac23} \}$, we have, with $V(T/2) = (v,\d_t v)(T/2)$, 
\bna
\|c \,  v\|_{L^2((0,T)\times M)} & \leq & \|c\|_{L^\infty}\|\,  v\|_{L^2((0,T) \times M)} \\
& \leq & C_{\delta} \|c\|_{L^\infty} e^{\kappa \mu} \nor{v}{L^2((0,T); H^1(\omega))} 
+ \frac{C_{\delta}\|c\|_{L^\infty} }{\mu}\left(\nor{V(T/2)}{\H^1(M)}+\nor{c u}{L^2((0 ,T) \times M)}\right).
\ena
So, for $\mu \geq 2C_{\delta}\|c\|_{L^\infty}$, we obtain
\bna
\|c \,  v\|_{L^2((0,T)\times M)} \leq C_{\delta} \|c\|_{L^\infty} e^{\kappa \mu} \nor{v}{L^2((0,T); H^1(\omega))} 
+ \frac{C_{\delta}\|c\|_{L^\infty} }{\mu} \nor{V(T/2)}{\H^1(M)} .
\ena
Plugging this into~\eqref{eq: obs HF-pot-bis} yields
\bna
\E_1(V(T/2)) \leq 
C_\delta^2 (1+\|c\|_{L^\infty})^2 \left(e^{2\kappa \mu}  \int_{0}^T \|  v (t)\|_{H^1(\omega)}^2 dt 
+ \frac{1}{\mu^2}\E_1(V(T/2))\right)  .
\ena
We now take $\mu€ = \max\{\mu_0  , \mu_0 \|c\|_{L^\infty}^{\frac23}, 2C_{\delta}\|c\|_{L^\infty}, \sqrt{2}C_\delta (1+\|c\|_{L^\infty}) \}$ so that to absorb the last term in the right handside, and finally obtain
\bna
\E_1(V(T/2)) \leq  
 C e^{C_{\delta} \|c\|_{L^\infty}}  \int_0^T \|  v (t)\|_{H^1(\omega)}^2 dt .
\ena
 Using now e.g. \eqref{NRJsqrtH1} implies $\E_1(V(0)) \leq C e^{\|c\|_{L^\infty}^\frac12} \E_1(V(T/2))$, which concludes the proof of the theorem.
\enp

\appendix
\section{Pseudodifferential calculus}
\label{s:pseudocalc}
\subsection{Remainder of elementary facts}
\label{s:pseudo-elementaire}
We define $S_{\phg}^{m}(T^* M)$, as the set of polyhomogeneous symbols of order $m$
on $M$. We recall that symbols in the class $S_{\phg}^{m}(T^* \R^n)$ behave well
with respect to changes of variables, up to symbols in $S_{\phg}^{m-1}(T^* \R^n )$ (see \cite[Theorem~18.1.17 and Lemma~18.1.18]{Hoermander:V3}).

We denote by $\Psi_{\phg}^{m}(M)$, the space of polyhomogeneous pseudodifferential operators of order
$m$ on $M$: one says that $A \in \Psi_{\phg}^{m}(M)$ if 
\begin{enumerate}
  \item its kernel $K_A \in \D'(M \times M)$ is smooth away from the diagonal $\Delta_{M} = \{ (x,x);\ x \in M\}$;
  \item for every coordinate patch $M_{\kappa} \subset M$ with
    coordinates $ M_{\kappa} \ni x \mapsto \kappa (x) \in
    \tilde{M}_{\kappa} \subset \R^{n}$ and all $\phi_0$, $\phi_1
    \in \Cinfc(\tilde{M}_{\kappa})$ the map 
    $$
    u \mapsto \phi_1 \big(\kappa^{-1}\big)^\ast A \kappa^\ast (\phi_0 u)
    $$
    is in $\Op(S_{\phg}^{m}(T^* \R^{n}))$.
\end{enumerate}

For $A \in \Psi_{\phg}^{m}(M)$, we denote by
$\sigma_m(A) \in S_{\phg}^{m}(T^* M)$ the
principal symbol of $A$ (see~\cite[Chapter 18.1]{Hoermander:V3}). Note
that the principal symbol is uniquely defined in $S_{\phg}^{m}(T^* M)$ because of the polyhomogeneous structure (see the remark following Definition~18.1.20 in~\cite{Hoermander:V3}).  Also, the map $\sigma_m : \Psi_{\phg}^{m}(M)\to S_{\phg}^{m}(T^* M)$ is onto (it suffices to construct a quantization on $T^* M$ by means of local charts, see for instance the discussion after Definition 18.1.20 in \cite{Hoermander:V3}).

\bigskip
  At places we shall need to consider pseudodifferential
  operators acting on $M$ yet depending upon the parameter $t \in
  (0,T)$ with some smoothness with respect to $t$. 
  Here, we follow \cite{DLRL:13} for the definitions and notation.
 Let $k \in \N \cup
  \{ \infty\}$, we say that $A_t \in \Con^k\big((0,T),
  \Op(S_{\phg}^m(\R^n\times \R^n))\big)$ if $A_t = \Op(a_t)$ with $a_t \in
  \Con^k((0,T), S_{\phg}^m(\R^n\times \R^n))$.  Next we say that $A_t
  \in \Con^k((0,T), \Psi_{\phg}^m(M))$ if
  \begin{enumerate}
  \item its kernel $K_{A_t}(x,y)$ is in $\Con^k\big((0,T); \Con^\ell (M \times M \setminus \Delta_M)\big)$ for all $\ell \in \N$; 
  \item for every coordinate patch $M_\kappa \subset M$ with
    coordinates $M_{\kappa} \ni x \mapsto \kappa(x) \in
    \tilde{M}_{\kappa} \subset \R^{n}$ and all $\phi_0$, $\phi_1
    \in \Cinfc(\tilde{M}_{\kappa})$ the map 
    $$
    u \mapsto \phi_1 \big(\kappa^{-1}\big)^\ast A_t \kappa^\ast (\phi_0 u)
    $$
    is in $\Con^k\big((0,T), \Op(S_{\phg}^m(T^* \R^n))\big)$.
\end{enumerate}

\bigskip
Let us now recall some basic facts concerning the first order hyperbolic Cauchy problem.
The following result can be adapted from~\cite[Chapter~XXIII]{Hoermander:V3}.
\begin{theorem}
\label{t:hormander-hyp}
Let $\mathcal{I}\subset \R$ be a compact interval and take $\sigma \in \R$. Assume $H(t) \in \Con^0(\mathcal{I}; \Psi^1_{\phg}(M))$ has real principal symbol. Then, there exists $C>0$ such that for all $f \in L^1(\mathcal{I};H^\sigma (M) )$, all $s \in \mathcal{I}$ and all $u_0 \in H^\sigma (M)$, the Cauchy problem
\bnan
\label{e:hyp-cauch-pb}
\begin{cases}
 \d_t u (t) - i H(t) u(t) =f(t), \qquad t \in \mathcal{I} , \\
 u|_{t = s} = u_0 .
 \end{cases}
\enan
has a unique (distribution) solution $u \in \Con^0(\mathcal{I} ;H^\sigma (M))$, that satisfies
\bna
\| u \|_{L^\infty(\mathcal{I} ;H^\sigma (M))} \leq C \| u_0 \|_{ H^\sigma (M)} +C \| f \|_{L^1(\mathcal{I} ;H^\sigma (M))} .
\ena
If moreover $f=0$, then  $u \in \Con^1(\mathcal{I} ;H^{\sigma-1} (M))$.
\end{theorem}
The constant $C$ essentially depends on a uniform bound on $\|H(t) - H^*(t) \|_{L^\infty(\mathcal{I} ;\L(H^\sigma (M)))}$ and commutator estimates.
The fact that $C$ does not depend on the initial time $s$ follows from the proof of~\cite[Lemma~23.1.1]{Hoermander:V3}.

Note also that, in case $f=0$, the regularity $\Con^1(\mathcal{I} ;H^{\sigma-1} (M))$ of the solution $u$ implies that~\eqref{e:hyp-cauch-pb} is in fact an equality of functions  in $\Con^0(\mathcal{I} ;H^{\sigma-1} (M))$.

\medskip

As a consequence of this theorem, for all $t,s \in \mathcal{I}$, there is a bounded linear solution map $S(t,s) \in \L(H^\sigma (M))$ (for any $\sigma \in \R$), given by $u_0 \mapsto u(t)$, where $u$ is the unique solution to~\eqref{e:hyp-cauch-pb} with $f=0$. We recall that the space $\Bo( I ;\L(B_1; B_2))$ is defined in Definition~\ref{def:Bo}.
As a consequence of~Theorem~\ref{t:hormander-hyp}, the solution operator $S(t,s)$ enjoys in particular the following regularity properties.

\begin{corollary}
\label{cor:S(t,s)}
With the notations and assumptions of Theorem~\ref{t:hormander-hyp}, we have
\begin{enumerate}
\item $S(t , s) \in  \Bo(\mathcal{I}\times \mathcal{I} ;\L(H^\sigma (M)))$ for all $\sigma \in \R$;
\item the linear operator $\d_t S(t , s) : u_0 \mapsto  \d_t \big( S(t , s) u_0 \big)$ satisfies $\d_t S(t , s) \in  \Bo(\mathcal{I}\times \mathcal{I} ;\L(H^\sigma (M); H^{\sigma - 1}(M)))$ for all $\sigma \in \R$ together with 
$\d_t S(t,s) - iH(t)S(t,s) = 0 $, $S(s, s) = \id$;
\item we have $S(t , s)S(s , t) = \id$ for all $(s , t) \in \mathcal{I}\times \mathcal{I}$;
\item for all $u_0 \in H^\sigma (M)$ and $t\in \mathcal{I}$, the application $s \mapsto S(t , s) u_0$ is in $\Con^0(\mathcal{I} ; H^{\sigma}(M)) \cap \Con^1(\mathcal{I} ; H^{\sigma - 1}(M))$ and, defining the linear operator $\d_s S(t , s) : u_0 \mapsto  \d_s \big( S(t , s) u_0 \big)$, it satisfies $\d_s S(t , s) \in  \Bo(\mathcal{I}\times \mathcal{I} ;\L(H^\sigma (M); H^{\sigma - 1}(M)))$ for all $\sigma \in \R$ together with 
$\d_s S(t,s) + iS(t,s)H(s) = 0$.
\end{enumerate}
\end{corollary}
Points  (1), (2) and (3) are direct consequences of Theorem~\ref{t:hormander-hyp}. Beware that $\d_t S(t , s)$ is not a derivative in the Banach space $\L(H^\sigma (M); H^{\sigma - 1}(M))$.
Point (4) follows from point (3) and the regularity properties of $S(t,s)$ with respect to $t$ (given in points (1) and (2)). The equation satisfied by $\d_s S(t,s)$ comes from the fact that $\d_2 S(t,s) S(s,t)= - S(t,s) \d_1 S(s,t)$ (where $\d_1$ and $\d_2$ stand for derivatives with respect to the first and second variables respectively).

\medskip
Note also that we have, for any $v \in \Con^0(\mathcal{I} ; H^{\sigma}(M)) \cap \Con^1(\mathcal{I} ; H^{\sigma - 1}(M))$ the formula:
$$
\d_t(S(t,s)v(t)) = \d_tS(t,s) v(t)  + S(t,s)\d_tv(t)  .
$$

\subsection{A non-autonomous non-selfadjoint Egorov theorem}
\label{s:egorov}
In the main part of the paper, we use the following non-selfadjoint non-autonomous version of the Egorov theorem. A semiclassical version of such a result in the autonomous case can be found in~\cite{Royer:10,Royer:these}.
\begin{theorem}
  \label{theorem: Egorov}
 Let $T>0$ and $H(t) \in \Cinf(0,T; \Psi^1_{\phg}(M))$ having real principal symbol $a_1 \in \Cinf(0,T; S^1_{\phg}(T^* M))$. Denote by $A_1(t) := \frac12 (H(t) +H(t)^*)\in\Cinf(0,T; \Psi^1_{\phg}(M))$  (the adjoints are taken in $L^2(M)$) and $A_0(t) := \frac{1}{2i} (H(t) - H(t)^*) \in\Cinf(0,T; \Psi^0_{\phg}(M))$, both selfadjoint for all $t \in [0,T]$, that satisfy $H(t)=A_1(t) + i A_0(t)$. 
 Both $a_1 = \sigma_1(A_1) \in \Cinf(0,T; S^1_{\phg}(T^* M))$, and  $a_0 = \sigma_0(A_0) \in \Cinf(0,T; S^0_{\phg}(T^* M))$ are real valued functions.
Denote by $S(t,s)$ the solution operator associated to $\d_t - i H(t)$, that is $S(s',s) u_0 = u(s')$ where 
  $$
 \d_t u (t) - i H(t) u(t) =0, \qquad u|_{t = s} = u_0 .
 $$ 
 Then, for any $P_m (s)\in  \Cinf  \big((0,T), \Psi^m_{\phg}(M) \big)$, $m \in
 \R$, there exist $Q(t,s) \in \Cinf  \big((0,T)^2, \Psi^m_{\phg}(M)\big)$ 
 and 
 \bna
 R(t,s) \in \Bo \big((0,T)^2, \L(H^\sigma(M), H^{\sigma +1-m}(M))\big) 
\\
\d_t R(t,s) , \d_s R(t,s)  \in \Bo \big((0,T)^2, \L(H^\sigma(M), H^{\sigma -m}(M))\big)
 \ena
for all $\sigma \in \R$, such that we have 
 $$
S(s,t)^*  P_m(s) S(s,t)  - Q(t,s) = R(t,s) , \qquad (t,s) \in (0,T)^2 .
 $$ 
Moreover, the principal symbol of $Q(t,s)$ is given by 
\bnan
\label{e:symbol-egorov}
 q(t,s , \rho)= p_m(s, \chi_{s,t}(\rho) ) e^{2\int_s^t a_0 (\tau, \chi_{\tau,t}(\rho) ) d\tau }\quad \in
 \Cinf((0,T)^2, S^m_{\phg}(T^\ast M))
 \enan
 where $p_m(s, \cdot) = \sigma_m( P_m(s))$, and $\chi_{s,t}(\rho_0) =\rho(s,t)$ is given by the flow of the
 Hamiltonian vector field associated with~$-a_{1}(s)$:
 $$
 \frac{d}{d s} \rho(s,t) = H_{-a_{1}(s)} (\rho(s,t)), \qquad \rho(t,t) = \rho_0 \in T^*M. 
 $$
\end{theorem}

The proof is inspired from~\cite[Chapter~7.8]{Taylor:vol2} and~\cite[Th\'eor\`eme~3.43]{Royer:these}.

\begin{remark}
In this result, the error term $R(t,s)$ is $1$-smoothing. Of course, a classical inductive construction (see~\cite[Section~18.1]{Hoermander:V3}) allows to replace this by an infinitely smoothing operator. This is not needed in the present paper since we only carry an analysis at first order.
\end{remark}

\begin{remark}
\label{rkegorovsimple}
In the simplest case $H=\Lambda$, we have
\begin{itemize}
\item $a_1=\lambda=|\xi|_{x}$;
\item $a_0=0$ because $\Lambda$ is selfadjoint;
\item $S(t,s)=e^{i(t-s)\Lambda}$ and hence $S(s,t)=e^{i(s-t)\Lambda}$ and $S(s,t)^*=(e^{i(s-t)\Lambda})^*=e^{i(t-s)\Lambda}$;
\item $\rho(s,t)=\varphi^-_{(s-t)}(\rho_0)=\varphi^+_{(t-s)}(\rho_0)$.
\end{itemize}
The conclusion of the Theorem (written with $s=0$ and $P_m$ independent on $s$) is therefore the classical result that $e^{it\Lambda} P_m e^{-it\Lambda}$ is (modulo a $1$-smoothing operator) a pseudodifferential operator of order $m$ with principal symbol $q(t,\rho)=p_m(\varphi^+_{t}(\rho))$.
\end{remark}
\bnp
First notice that $S(t,s)$ (solution operator at time $t$, issued from $s$) satisfies
$$
\d_t S(t,s) - iH(t)S(t,s) = 0 , \quad S(s, s) = \id .
$$
As a consequence, since $S(t,s)S(s,t) = \id$, we also have, with $H(t)^* = A_1(t) - i A_0(t)$,
\bna
\d_t S(s,t) + iS(s,t)H(t) = 0 ,  \\ 
\d_t S(t,s)^* + iS(t,s)^* H(t)^*= 0 ,\\ 
\d_t S(s,t)^* - iH(t)^* S(s,t)^* = 0 .
\ena
Corollary~\ref{cor:S(t,s)} yields the following regularity properties
\bna
\label{e:regS}
S(t , s) \in  \Bo (\mathcal{I}\times \mathcal{I} ;\L(H^\sigma (M) )) , \quad 
\d_t S(t , s) , \d_s S(t,s) \in  \Bo (\mathcal{I}\times \mathcal{I} ;\L(H^\sigma (M); H^{\sigma - 1}(M)))
\ena
as well as for $S(t,s)^*$, for all $\sigma\in \R$. 

Now, setting
 $$
P(t,s) := S(s,t)^*  P_m(s) S(s,t)  ,
 $$
and using the above equations, we have $P(s,s)=P_m(s)$ with 
\bnan
\label{eqnP}
\d_tP(t ,s) = i H(t)^* P(t,s) - i P(t,s)H(t)  = i [A_1(t) ,P(t,s)] + A_0(t) P(t,s) + P(t,s) A_0(t) .
\enan
We now construct an approximate pseudodifferential solution $Q(t,s)$ for~\eqref{eqnP}: its principal symbol $q(t,s, x, \xi)$ should satisfy
\bnan
\label{e:transp-symbol}
\d_t q(t,s, \cdot) =  \{ a_1(t, \cdot) , q(t,s, \cdot)  \} + 2 a_0 (t,\cdot) q(t,s, \cdot)  
 , \quad \text{and} \quad  q(s,s, \rho) = p_m(s, \rho) ,
\enan
where $\{ \cdot , \cdot\}$ stands for the Poisson bracket in the $(x, \xi)$ variables.

We first check that the function $q(t,s, x, \xi)$ defined in~\eqref{e:symbol-egorov} satisfies~\eqref{e:transp-symbol}. From~\eqref{e:symbol-egorov}, and using $\chi_{\tau,t} \circ \chi_{t,s}(\rho) =\chi_{\tau,s} (\rho)$, we have:
\bna
 q(t,s , \chi_{t,s}(\rho) )= p_m(s,  \rho ) e^{2\int_s^t a_0 (\tau, \chi_{\tau,s}(\rho) ) d\tau } .
  \ena
This yields $q(s,s, \rho) = p_m(s, \rho)$ and  
\bna
 \d_t \left( q(t,s , \chi_{t,s}(\rho) ) e^{-2\int_s^t a_0 (\tau, \chi_{\tau,s}(\rho) ) d\tau } \right)= 0 ,
  \ena
which, according to the definition of the flow $\chi_{t,s}$,  is 
\bna
\Big(  (\d_t q)(t,s ,\cdot ) +  \{ - a_1(t, \cdot) , q(t,s, \cdot)  \} - 2 a_0 (t,\cdot ) q(t,s, \cdot )  \Big) \big( \chi_{t,s}(\rho) \big)  e^{-2\int_s^t a_0 (\tau, \chi_{\tau,s}(\rho) ) d\tau } = 0 ,
  \ena
for all $(t,s) \in (0,T)^2$ and $\rho \in S^*M$, which proves \eqref{e:transp-symbol}.

Note that we use the homogeneity of $a_1$ of order $1$ allows to keep the homogeneity of $q(t,\rho)$. This allows to select one $Q(t, s)$, so that
\bnan
\label{e:defQ2}
Q(t, s) \in \Cinf \big((0,T)^2, \Psi^m_{\phg}(M)\big) \text{ satisfies  } \sigma_m(Q(t,s)) = q(t, s, \cdot ) .
\enan

From~\eqref{e:transp-symbol} and pseudodifferential calculus, we now have
\bnan
\label{eqnQ}
\d_t Q(t,s) & = & i [A_1(t) ,Q(t,s)] + A_0(t) Q(t,s) + Q(t,s) A_0(t) + R(t,s) \nonumber \\
&= & i H(t)^* Q(t,s) - i Q(t,s)H(t) + R(t,s),
\enan
with $R \in  \Cinf((0,T)^2; \Psi^{m-1}_{\phg}(M))$.
We now estimate the remainder $Q(t,s) - P(t,s)$. We set
 $$T(t,s):=S(t,s)^* \big(Q(t,s) - P(t,s)\big)S(t,s) = S(t,s)^* Q(t,s)S(t,s) - P_m(s),$$
so that we have
\bna
\d_t T(t,s) & =&  \d_t \big( S(t,s)^* Q(t,s)S(t,s) \big) \\
& = & S(t,s)^* \big( -i H(t)^* Q(t,s) + \d_t Q(t,s) + i Q(t,s) H(t) \big) S(t,s) \\ 
& = & S(t,s)^*  R(t,s) S(t,s) ,
\ena
after having used~\eqref{eqnQ}. This yields
$$
Q(t,s) - P(t,s) = S(s,t)^* \left(Q(s,s)-P_m(s) + \int_s^tS(t',s)^*  R(t',s) S(t',s) dt' \right) S(s,t) ,
$$
where $R \in  \Cinf((0,T)^2; \Psi^{m-1}_{\phg}(M))$ and $Q(s,s)-P_m(s) \in \Cinf( (0,T) ;  \Psi^{m-1}_{\phg}(M))$.
This now implies
\bna
Q(t,s) - P(t,s)  \in \Bo \big((0,T)^2, \L(H^\sigma(M), H^{\sigma +1-m}(M))\big) , \\
\d_t \big(Q(t,s) - P(t,s) \big) , \d_s \big(Q(t,s) - P(t,s) \big) \in \Bo  \big((0,T)^2, \L(H^\sigma(M), H^{\sigma -m}(M))\big),
\ena
for any $\sigma \in \R$. This, together with the expression of $Q$ in \eqref{e:defQ2} concludes the proof of the theorem.

\enp

\subsection{Smoothing properties of some operators}
\label{s:smoothing-prop}
The following lemma is taken from~\cite[Lemma~A.1]{DLRL:13} and inspired by~\cite{DL:09}. 
\begin{lemma}
\label{lemma: regularity Hs}
Let $\gamma ,\delta \in \R$ such that $\gamma \neq \delta$, and $B_0 \in \Psi_{\phg}^0(M)$. Then, the operator defined by
\begin{align*}
 B (T) = \int_0^T e^{-i t \gamma \Lambda} B_0 e^{i t \delta \Lambda} d t ,
\end{align*}
satisfies $B \in \Bo_{\loc}(\R ;\L(H^{\sigma}(M), H^{\sigma+1}(M)))$ for all $\sigma \in \R$.
\end{lemma}
This lemma suffices for the study of the Klein Gordon equation in Section~\ref{s:model-case}.
In the general case of Section~\ref{s:generalcase-T} however, we need the following non-autonomous version of this result.
\begin{lemma}
\label{lemma: regularity Hs non self}
Let $\mathcal{I} \subset \R$ be an interval, and let $H_+ , H_- \in \Cinf(\mathcal{I};\Psi^1_{\phg} (M))$ such that $\lambda = \sigma_1(H_+) = - \sigma_1(H_-)  \in \R$ is time independant and elliptic. 
Then for any $B_0 \in \Psi_{\phg}^m(M)$, $m \in \R$, the operator defined by
\begin{align*}
B(T) = \int_0^T  S_+(t,0)^* B_0 S_-(t,0) d t ,
\end{align*}
where $S_\pm (t,0)$ is the solution operator for the evolution equation $\d_t - i H_\pm(t)$,
satisfies for all $\sigma  \in \R$, $B \in \Bo_{\loc}(\mathcal{I} ;\L(H^{\sigma}(M), H^{\sigma+1-m }(M)))$.
\end{lemma}

We refer to Corollary~\ref{cor:S(t,s)} for the properties of $S_\pm(t,0)$. We shall need the following lemma.
\begin{lemma}
\label{l:comm-evol-lambda}
Let $\mathcal{I} \subset \R$ be an interval, and let $H(t) \in \Cinf(\mathcal{I};\Psi^1_{\phg} (M))$ with real principal symbol and denote by $S  (t,0)$ the solution operator for the evolution equation $\d_t - i H(t)$. Then, for any $A \in  \Psi^m_{\phg} (M)$, we have
\bnan
\label{e:commutator-group}
[A  ,  S(t,0)] = \int_0^t S(t,s) [A ,i H(s)] S(s,0) ds .
\enan
In particular if $A = \Lambda$ and $H(t) = \Lambda + iR(t)$, with $R \in  \Cinf(\mathcal{I};\Psi^0_{\phg} (M))$, we have $$[\Lambda ,  S(t,0)] , [\Lambda ,  S(t,0)^*] \in \mathcal{B}_{\loc}(\mathcal{I} ;\L(H^{s}(M)))$$ for all $s \in \R$.
\end{lemma}
\bnp[Proof of Lemma~\ref{l:comm-evol-lambda}]
The function $u(t) = [A  ,  S(t,0)] u_0 = A S(t, 0)u_0 - S(t,0)A u_0$ satisfies $u(0) =0$ and solves
\bna
\d_t u(t) = A i H(t) S(t, 0)u_0 - i H(t) S(t,0)A u_0 = [A, iH(t)]S(t,0) u_0  + i H(t) u(t) , 
\ena
so that the Duhamel formula directly yields~\eqref{e:commutator-group}.
\enp

\bnp[Proof of Lemma~\ref{lemma: regularity Hs non self}]
We first notice that $B(T) \in \Bo_{\loc}(\mathcal{I} ;\L(H^{s}(M), H^{s-m}(M)))$ since $S_\pm(t,0)$ preserve regularity. We recall also that
\bnan
\label{e:Spmeq}
\d_t S_\pm(t,0) - iH_\pm(t)S_\pm(t,0) = 0 , \quad \d_t S_\pm(t,0)^* + iS_\pm(t,0)^* H_\pm(t)^*= 0 ,
\enan 
To prove the result, it suffices to prove that $\Lambda B(T) \in \Bo_{\loc}(\mathcal{I} ;\L(H^{s}(M), H^{s -m }(M)))$. We thus compute
\bna
i \Lambda B(T)  = \int_0^T i S_+(t,0)^* \Lambda B_0 S_-(t,0) d t + \int_0^T i  [S_+(t,0)^* , \Lambda ] B_0 S_-(t,0) d t . 
\ena
The second term belongs to $\Bo_{\loc}(\mathcal{I} ;\L(H^{s}(M), H^{s -m }(M)))$ according to Lemma~\ref{l:comm-evol-lambda}. The first term may be rewritten as
\bna
&&\int_0^T i S_+(t,0)^* \Lambda B_0 S_-(t,0) d t \\
&& \qquad \qquad \qquad \qquad= \int_0^T i  S_+(t,0)^*H_+(t)^* B_0 S_-(t,0) d t + \int_0^T i  (\Lambda - H_+(t)^*) S_+(t,0)^* B_0 S_-(t,0) d t .
\ena
The second term belongs to $\Bo_{\loc}(\mathcal{I} ;\L(H^{s}(M), H^{s -m }(M)))$ since $\Lambda - H_+(t)^* \in\Cinf(\mathcal{I};\Psi^0_{\phg} (M))$, and it remains only to examine the first one. Using~\eqref{e:Spmeq}, we now have, for some $R \in \Bo_{\loc}(\mathcal{I} ;\L(H^{s}(M), H^{s -m }(M)))$, 
\bna
i \Lambda B(T) 
& = &\int_0^T - \d_t S_+(t,0)^*  B_0 S_-(t,0) d t +R \\
& =  &\int_0^T  S_+(t,0)^*  B_0  \d_tS_-(t,0) d t - \left[S_+(t,0)^*  B_0 S_-(t,0)  \right]_0^T + R ,
\ena
after an integration by parts (note that this is done in the weak sense, i.e. when applied to a function). Using again~\eqref{e:Spmeq}, we obtain, for other remainders $R \in \Bo_{\loc}(\mathcal{I} ;\L(H^{s}(M), H^{s -m }(M)))$,
\bna
i \Lambda B(T) 
 & =  & \int_0^T  S_+(t,0)^*  B_0  i H_-(t)S_-(t,0) d t + R , \\
 & =  & \int_0^T  S_+(t,0)^*  B_0  i(- \Lambda) S_-(t,0) d t + R , 
\ena
where we used that $- \Lambda - H_-(t) \in \Cinf(\mathcal{I};\Psi^0_{\phg} (M))$. Using $[B_0  , \Lambda] \in \Psi^m_{\phg} (M)$, we now have
\bna
i \Lambda B(T)  & =  & \int_0^T  S_+(t,0)^*  (- i \Lambda) B_0  S_-(t,0) d t + R , 
\ena
that is, using again Lemma~\ref{l:comm-evol-lambda},
\bna
i \Lambda B(T) =    - i \Lambda B(T)  + R , 
\ena
with $R \in  \Bo_{\loc}(\mathcal{I} ;\L(H^{s}(M), H^{s -m }(M)))$. This concludes the proof of the lemma.
\enp

\subsection{Uniform estimates on compact manifolds}
\label{s:unif-pseudo}
We give here a version of the sharp G{\aa}rding inequality (and also boundedness estimates for pseudodifferential operators) on a compact manifold, with a uniform dependence of the constant \wrt the operator involved. Its counterpart on $\R^n$ (of which the result presented here is a consequence) is given in \cite[Theorem~2.5.4]{Lerner:10} for instance.

We use the notation $M_\eps = \{(x,y)\in M\times M , \dist(x,y)>\eps\}$.
\begin{theorem}
\label{th: sharp garding manifold}
Let $(U_j ,\kappa_j)_{j = 1...N}$ be a fixed atlas of $M$ and $(\psi_j)_{j = 1...N}$ a subordinated partition of unity. Let $\tilde{\psi}_j \in \Cinfc(U_j)$ be such that $\tilde{\psi}_j = 1$ on $\supp(\psi_j)$.
Then,  for all $s \in \R$, there exists $\gamma$ a seminorm on $S^0_{\phg}(T^* \R^n)$, there exist $\eps>0$, $\ell >0$ and $C>0$ such that, for all $A \in \Psi^0_{\phg}(M)$ and all $u \in H^s(M)$, we have
\begin{align}
\label{e:bounded-unif}
\|A u \|_{H^s(M)} \leq C \left( \sup_{j \in \{1...N\}}\gamma(a^j) + \|K_A\|_{W^{\ell, \infty}(M_\eps)}
\right) \|u\|_{H^s(M)} , 
\end{align}
and, if  moreover $\sigma_0(A)\geq 0$ on $T^*M$, 
\begin{align}
\label{e:garding-unif}
\Re (A u , u)_{H^s(M)} \geq - C \left( \sup_{j \in \{1...N\}}\gamma(a^j) + \|K_A\|_{W^{\ell, \infty}(M_\eps)}
\right) \|u\|_{H^{s-1/2}(M)}^2 , 
\end{align}
where $a^{j} \in S^0_{\phg}(T^* \R^n)$ is the (full) symbol of the operator $(\kappa_j^{-1})^* \psi_j A \tilde{\psi}_j \kappa_j^* \in \Psi^0_{\phg}(\R^n)$.
\end{theorem}

As a direct consequence, we have the following corollary.

\begin{corollary}
\label{cor:unif-bound}
Let $s \in \R$, $T_1<T_2$ and assume $A_t \in \Con^0([T_1, T_2]; \Psi^0_{\phg}(M))$. Then, there exists a constant $C>0$ such that
\begin{align*}
\|A_t u \|_{H^s(M)} \leq  C \|u\|_{H^{s}(M)} , \quad \text{for all } t \in [T_1, T_2], \text{ and } u \in H^s(M), 
\end{align*}
and, if  moreover $\sigma_0(A)\geq 0$ on $[T_1, T_2] \times T^*M$, 
\begin{align*}
\Re(A_t u , u)_{H^s(M)} \geq - C \|u\|_{H^{s -1/2}(M)}^2 , \quad \text{for all } t \in [T_1, T_2], \text{ and } u \in H^s(M).
\end{align*}
\end{corollary}

\begin{proof}[Proof of Theorem~\ref{th: sharp garding manifold}]
We only prove the uniform G{\aa}rding inequality~\eqref{e:garding-unif}. The proof of the uniform boundedness estimate~\eqref{e:bounded-unif} is the same (using e.g.~\cite[proof of Theorem~1.1.4]{Lerner:10}).

Notice first that the result in $H^s(M)$ is a consequence of the result in $L^2(M)$ and~\eqref{e:bounded-unif}: For $u \in \Cinf(M)$, applying~\eqref{e:garding-unif} to $\Lambda^s u$ yields 
\begin{align*}
(A \Lambda^s u , \Lambda^s u)_{L^2(M)} \geq - C_0 \|\Lambda^s u\|_{H^{-1/2}(M)}^2 =- C_0 \| u\|_{H^{s-1/2}(M)}^2 , 
\end{align*}
with $C_0 =   \sup_{j \in \{1...N\}}\gamma(a^j) + \|K_A\|_{W^{\ell, \infty}(M_\eps)}$. Writing now 
\begin{align*}
|(A \Lambda^s u , \Lambda^s u)_{L^2(M)} - (\Lambda^s A  u , \Lambda^s u)_{L^2(M)} |=
| (\Lambda^{1/2} [A , \Lambda^s] u , \Lambda^{s-1/2} u)_{L^2(M)} | \leq C_A \| u\|_{H^{s-1/2}(M)}^2 ,
\end{align*}
where the constant $C_A$ has the same form as $C_0$ according to~\eqref{e:bounded-unif}, yields the result in $H^s(M)$. 

Let us now prove the case $s=0$. 
We have 
\begin{align}
\label{eq: decomposition A}
A = \sum_{j = 1...N} \psi_j A 
= \sum_{j = 1...N} \psi_j A \tilde{\psi}_j  + \psi_j A (1 - \tilde{\psi}_j ).
\end{align}
The kernel of each operator $\psi_j A (1 - \tilde{\psi}_j )$ is given by $K^{j}(x,y):=\psi_j(x) K_A(x,y) (1 - \tilde{\psi}_j (y))$. Since $\psi_j (1 - \tilde{\psi}_j ) = 0$, it is supported in the set $M_{\eps^{j}}$ for some $\eps^{j}>0$. As a consequence, this operator is infinitely smoothing and we have in particular 
\begin{align*}
\| \psi_j A (1 - \tilde{\psi}_j ) \|_{\L(H^{-1/2}(M) ; H^{1/2}(M) )} \leq C_j \|K_A\|_{W^{\ell, \infty}(M_{\eps^j})},
\end{align*}
so that  
\begin{align}
\label{eq: KA smoothing}
\big| \big( \psi_j A (1 - \tilde{\psi}_j )u ,u  \big)_{L^2(M)} \big| 
& \leq 
\|\big( \psi_j A (1 - \tilde{\psi}_j )u\|_{H^{1/2}(M)} \|u\|_{H^{-1/2}(M)} \nonumber \\
& \leq C_j \|K_A\|_{W^{\ell, \infty}(M_{\eps^j})}\|u\|_{H^{-1/2}(M)}^2 .
\end{align}
Next, concerning the terms of the form $\psi_j A \tilde{\psi}_j$ in \eqref{eq: decomposition A}, we write 
\begin{align*}
\big( \psi_j A \tilde{\psi}_j u ,u  \big)_{L^2(M)} = \big( \big((\kappa_j^{-1})^* \psi_j A \tilde{\psi}_j \kappa_j^* \big) (\kappa_j^{-1})^*u , (\kappa_j^{-1})^*u \big)_{L^2(\R^n, \sqrt{\det(g)}dL)} ,
\end{align*}
where the principal symbol of the operator $(\kappa_j^{-1})^* \psi_j A \tilde{\psi}_j \kappa_j^*$ is $(\kappa_j^{-1})^* (\psi_j \sigma_0(a)) \geq 0$ on $T^* \R^n$. Using the sharp G{\aa}rding inequality in $\R^n$ as stated in \cite[Theorem~2.5.4]{Lerner:10}, we obtain, for smooth compactly supported functions $v$, 
\begin{align}
\label{eq: positive Aj}
\big( \big((\kappa_j^{-1})^* \psi_j A \tilde{\psi}_j \kappa_j^* \big) v , v \big)_{L^2(\R^n ,\sqrt{\det(g)}dL)}
\geq - C_j \gamma(a^j)\|v\|_{H^{-1/2}(\R^n)}^2 .   
\end{align}
Note that we have used here that the sharp G{\aa}rding inequality remains unchanged under the addition of an operator in $\Psi^{-1}_{\phg}(\R^n)$.

Finally, combining~\eqref{eq: KA smoothing}, \eqref{eq: positive Aj}, with \eqref{eq: decomposition A}, and recalling that there is a finite number of coordinate patches $U_j$, we obtain the result of Theorem~\ref{th: sharp garding manifold}.
\end{proof}

\section{Geometric facts}
\label{app:geom}
\subsection{Definitions and notations}
Recall that $M$ is a compact manifold without boundary, that for $x \in M$, $T_x M$ denotes the tangent space to $M$ at the point $x$, and $T^*_x M$ its dual space, the cotangent space to $M$ at $x$. We also denote $\pi :TM \to M$ and $\pi :T^*M \to M$ the canonical projections to the manifold, the duality bracket at $x$ being denoted by $\langle \cdot , \cdot \rangle_{x} =\langle \cdot , \cdot \rangle_{T_x^* M , T_x M}  $. 
The manifold $M$ is endowed with a Riemannian metric $g$, that is: for any $x \in M$, $g_x$ is a positive definite quadratic form on $T_x M$, depending smoothly on $x$. The Riemannian metric $g$ furnishes an isomorphism $T_x M \to T_x^* M$, $v \mapsto v^\flat :=g_x(v , \cdot)$, with inverse $v = (v^\flat)^\sharp$. The metric $g$ on $TM$ induces a metric $g^*$ on $T^*M$, canonically defined by $g^*_x(\xi ,\eta) = g_x(\xi^\sharp ,\eta^\sharp)$ for $x \in M$, and $\xi ,\eta \in T^*_x M$. We denote by $SM$ (\resp $S^*M$) the Riemannian sphere (\resp cosphere) bundle over $M$, with fiber over $x\in M$ given by $\{v \in TM ,g_x(v,v) =1\}$ (\resp $\{\xi \in T^*M ,g^*_x(\xi ,\xi ) =1\}$).

\bigskip
We define the Hamiltonian $\lambda(x,\xi) = |\xi|_x = \sqrt{g^*_x(\xi ,\xi)} \in \Cinf(T^*M \setminus {0})$, which is a homogeneous function of degree one. We denote by $H_{\lambda}$ and $\varphi_t=\varphi_t^+$ the associated Hamiltonian vector field and flow, that is 
$$ 
 \frac{d}{d t} \varphi_t(\rho) = H_{\lambda} ( \varphi_t(\rho) ), \qquad  \varphi_0(\rho)  = \rho  \in T^*M,
 $$
with, in local charts, $ H_{\lambda} = \d_\xi \lambda \cdot \d_x -  \d_x \lambda \cdot \d_\xi$. Writing $\varphi_t(\rho)=  (x(t) , \xi(t))$, we have, still in local charts, 
\bnan
\label{e:ham-curve}
\dot{x}(t) = \d_\xi \lambda(x(t), \xi(t)) , \quad \text{ and } \quad \dot{\xi}(t) = - \d_x \lambda(x(t), \xi(t)).
\enan
 This flow is globally defined, for it preserves the function $\lambda$ . In particular $(x(t) , \xi(t)) \in S^*M = \{(x,\xi) \in T^*M , \lambda(x,\xi)=1\}$ for all $t \in \R$ if $(x(0) , \xi(0)) \in S^*M$. The following lemma gives the link between geodesics and projections on $M$ of the curves of $\varphi_t$ (see for instance \cite[Theorem~2.124]{GHL04} in the case of $H_{\lambda^2/2}=\lambda H_{\lambda}$, which gives the same result up to a reparametrization of the curve $x(t)$).
\begin{lemma}
\label{l:geodesic-flow}
Let $I= [a,b] \subset \R$. A curve $(x(t) , \xi(t))_{t \in I}$ on $T^*M \setminus 0$ is a Hamiltonian curve of $\lambda$ (i.e. satisfies~\eqref{e:ham-curve}) 
in $T^*M\setminus 0$  if and only if the curve $(x(t))_{t \in I}$ on $M$ is a geodesic curve of the metric $g$ on $M$ (parametrized by arclength) such that
 $(x(t) , \dot{x}(t)) \in SM$, $t \in I$. 
 Moreover, we have $\dot{x}(t) =\frac{\xi(t)^\sharp}{|\xi(t)|_{x(t)}} \in S_{x(t)}M$, $t\in I$.
 \end{lemma}
In the main part of the article, we also use the Hamiltonian flow $\varphi_t^-$ associated to the Hamiltonian $-\lambda$ (which, as well, is global and preserves $S^*M$). Of course, it is linked with that of $\lambda$ according to the following lemma.
\begin{lemma}
\label{l:phi+phi-}
For all $t \in \R$ and $\rho \in T^*M$, we have $\varphi^-_t(\rho) = \varphi_{-t}(\rho)$. Moreover, denoting by $\sigma : T^*M \to T^*M$ the involution $(x, \xi) \mapsto (x, -\xi)$, we have $ \sigma \circ \varphi_{t}(\rho) =  \varphi_{-t}  \circ \sigma(\rho)$.
\end{lemma}
This is classical. The first fact is \eg a consequence of~\cite[Lemma~B.1]{DLRL:13}, and the second of~\cite[Lemma~B.3]{DLRL:13}.

\medskip
In the main part of the paper, we use the Riemannian distance to a subset $E \subset M$, defined by
\bna
\dist(x_1 , E) = \inf_{x_0 \in E} \dist(x_0 , x_1) ,
\ena
with
$$
 \dist(x_0 , \x_1) = \inf_{\gamma \in C^1([0,1] ; M), \gamma(0) = x_0 , \gamma(1) = x_1}\length(\gamma),
$$
where the Riemannian length of a path $\gamma \in C^1([0,1] ; M)$ is given by 
$\length(\gamma) = \int_0^1 \sqrt{ g_{\gamma(t)}(\dot{\gamma}(t), \dot{\gamma}(t)) } \ dt$.

\medskip
Given a smooth function $u$ on $M$, we define the vector field $\nabla u$ by $\nabla  u(x) =  \big( d u(x)\big)^\sharp$. As well, the Laplace-Beltrami operator may be defined by the identity
$$
\int_M (\Delta u)(x) v(x) dx =  - \int_M g_x( \nabla u(x), \nabla v(x)) dx , 
$$
where $dx$ is the Riemannian volume element (given by $\sqrt{\det(g)} dL(x)$ in local charts, where $dL(x)$ is the Lebesgue measure).

\medskip
To conclude this section, we give a proof of Lemma~\ref{l:egal-gpm-pde}, as consequence of Lemma~\ref{l:phi+phi-}.
As another corollary (which is a generalization of the former), we also have Lemma~\ref{l:egal-gpm-pde}, a proof of which we may now write.
\begin{proof}[Proof of Lemma~\ref{l:egal-gpm-pde}]
Recalling that $\varphi_{s}^-(\rho) = \varphi_{-s}^+(\rho) =\varphi_{-s}(\rho)$ according to Lemma~\ref{l:phi+phi-},~\eqref{e:def-gpm-pde} can be rewritten, using $(x (s) , \xi (s)) = \varphi_s  (x_0 ,\xi_0)$ for $s\in \R$ as 
$$
g_T^\pm (x_0 , \xi_0)=  \int_0^T  \obs^2(x(\pm t))  \exp \left(\int_0^t \Re(b_0)(\tau, x(\pm \tau)) \pm \left< \frac{\xi(\pm \tau)}{|\xi (\pm \tau)|_{x(\pm \tau)}} , \Re(b_1)(\tau , x(\pm \tau)) \right>_{x(\pm \tau)}  d\tau \right) dt.
$$
According to Lemma~\ref{l:phi+phi-}, we also have $\sigma \circ \varphi_s = \varphi_{-s} \circ \sigma$ (where $\sigma(x,\xi) = (x, -\xi)$), that is, denoting $(x (s, x_0, \xi_0) , \xi (s,x_0, \xi_0)) = \varphi_s  (x_0 ,\xi_0)$,
$$
x (-s, x_0, -\xi_0) = x (s,x_0, \xi_0)) , \quad \xi (-s, x_0,- \xi_0) = - \xi (s,x_0, \xi_0)) , \quad s \in \R , (x_0, \xi_0) \in T^*M .
$$
Plugging this into the expression of $g_T^-$, we obtain
\bna
g_T^- (x_0 , - \xi_0) 
& = & \int_0^T  \obs^2(x(-t ,x_0 , -\xi_0))  \exp \Bigg(\int_0^t \Re(b_0)(\tau, x(-\tau ,x_0 , -\xi_0)) \\
&&  - \left< \frac{\xi(-\tau ,x_0 , -\xi_0)}{|\xi (-\tau ,x_0 , -\xi_0)|_{x(-\tau ,x_0 , -\xi_0)}} , \Re(b_1)(\tau , x(-\tau ,x_0 , -\xi_0)) \right>_{x(-\tau ,x_0 , -\xi_0)}  d\tau \Bigg) dt \\
& = & \int_0^T  \obs^2(x(t ,x_0 ,\xi_0))  \exp \Bigg(\int_0^t \Re(b_0)(\tau, x(\tau ,x_0 ,\xi_0)) \\
&&  - \left< \frac{-\xi(\tau ,x_0 , \xi_0)}{|\xi (\tau ,x_0 ,\xi_0)|_{x(\tau ,x_0 ,\xi_0)}} , \Re(b_1)(\tau , x(\tau ,x_0 , \xi_0)) \right>_{x(\tau ,x_0 , \xi_0)}  d\tau \Bigg) dt \\
& = & g_T^+ (x_0 , \xi_0).
\ena
This is $g_T^-  \circ \sigma = g_T^+$.
\end{proof}
This Lemma contains in particular the following result,  after having used that $\sigma : S^*M \to S^*M$ is a bijection.
\begin{corollary}
\label{e:moy-=moy+}
For any function $f \in \Con^0(M)$, for any $T>0$, we have
$$
\min_{\rho \in S^*M} \int_0^T f \circ \pi \circ \varphi_t^+ (\rho)dt = \min_{\rho \in S^*M} \int_0^T f \circ \pi \circ \varphi_t^-(\rho) dt .
$$
\end{corollary}

\subsection{Comparing $T_{UC}(\omega)$ and $T_{GCC}(\omega)$}
\label{s:compare-time}
In this section, we briefly prove that  $2 \mathcal{L}(M ,\omega) =T_{UC}(\omega)\leq T_{GCC}(\omega)$ (where these quantities are defined in~\eqref{e:def-L}, \eqref{e:def-TUC} and~\eqref{e:def-TGCC} respectively) in general and study the case of equality.
\begin{lemma}
\label{lmTUleqTGCC}
We always have $T_{GCC}(E)\geq 2\mathcal{L}(M ,E)$.
\end{lemma}
\bnp
Let $\eps>0$, we prove that for any $x\in M$, $2\dist(x,E)\leq T_{GCC}(E)+2\eps$.

Fix $x\in M$. By definition, there exists $x_1\in E$ so that 
\bnan
\label{inegdist}
\dist(x,E)\leq d_1 : =\dist(x,x_1)\leq \dist(x,E)+\eps.
\enan
Take any $\xi\in S_x^* M$ and define the geodesic path $\gamma(t)=\pi \circ \varphi_t((x,\xi))$ for $t \in [0,d_1]$. According to Lemma~\ref{l:geodesic-flow}, we have
\bna
\length\big(\pi \circ \varphi_t((x,\xi))_{|[0,T]} \big) = T , \quad \text{for all } T>0 . 
\ena
Hence, we have $\gamma(t)\notin E$ for $t\in [0,d_1-\eps]$, otherwise we would have $\dist(x,E)\leq d_1-\eps$, which contradicts \eqref{inegdist}. The same arguments proves that if we define $\widetilde{\gamma}(t)=\pi \circ \varphi_t((x,-\xi))$ defined on $[0,d_1-\eps]$,  we have $\widetilde{\gamma}(t)\notin E$ for $t\in [0,d_1-\eps]$. Using Lemma~\ref{l:phi+phi-}, we also have $\widetilde{\gamma}(t)=\pi \circ \varphi_{-t}((x,\xi))$ on $[0,d_1-\eps]$.

The curve $t \mapsto \pi \circ \varphi_t((x,\xi))$ for $t \in [-d_1 +\eps ,d_1-\eps]$ is thus the concatenation of the two geodesics $\gamma$ and $\widetilde{\gamma}$. This is still a geodesic of length $2d_1-2\eps$ that does not intersect $E$. Therefore, we have $T_{GCC}(E)\geq 2d_1-2\eps$. The first part of \eqref{inegdist} gives $T_{GCC}(E)\geq 2\dist(x,E)-2\eps$. This gives the result.
\enp

\begin{remark}
\label{rem-boundary-T-T}
In the case $\d M \neq \emptyset$, we also have $T_{UC}(\omega)\leq T_{GCC}(\omega)$ for $\omega$ open subsets of $M$, as well as $T_{UC}(\Gamma)\leq T_{GCC}(\Gamma)$ for $\omega$ open subsets of $\d M$. The proof is similar, replacing $\varphi_t$ by the appropriate broken bicharacteristic flow (see~\cite{MS:78} or \cite[Chapter~XXIV]{Hoermander:V3}).

\end{remark}

\begin{lemma}[Equality case]
\label{lmTUeqTGCC}
Assume $T_{GCC}(E)=2\mathcal{L}(M ,E)=2R_0 >0$, then, there is $x_0\in M$ such that $\dist(x_0 , E) = R_0$ and for every $\xi\in S^*_{x_0} M$
\bna
&\pi \circ \varphi_t((x_0,\xi))\notin E&\quad \forall |t|< R_0\\
&\pi \circ \varphi_t((x_0,\xi))\in \overline{E}&\quad \forall |t|=R_0.
\ena
Moreover, these properties are also satisfied by any $x_0\in M$ such that $\dist(x_0 , E) = R_0$.

Finally, for any $x \in M$, we have the following alternative:
\begin{itemize}
\item either $\dist(x, E) < R_0$,
\item or $\dist(x, E) = R_0$ and the connected component of $M\setminus \overline{E}$ containing $x$ is the open ball $B(x, R_0)$.
\end{itemize}
\end{lemma}
\bnp
The function $x\mapsto \dist(x,E)$ is a continuous function on the compact manifold $M$. Consider $x_0$ one of the points where it takes its maximum $R_0=\dist(x_0,E)=\mathcal{L}(M ,E)$. 

For any $\xi \in S_{x_0}^* M$, we have necessarily $\pi \circ \varphi_t((x_0,\xi))\notin E$ $\forall |t|< R_0$, otherwise, we would have $\dist(x_0,E)< R_0$.

Moreover, assume that there exists $\xi_0 \in S_{x_0}^* M$ so that $\pi\circ \varphi_{R_0}((x_0,\xi_0))\notin \overline{E}$. By continuity of $t \mapsto \pi \circ \varphi_t((x_0,\xi_0))$ and the fact that the complementary of $\overline{E}$ is open, there exists $\eps>0$ so that $\pi \circ \varphi_t((x_0,\xi_0))\notin \overline{E}$ for $t\in ]R_0-\eps,R_0+\eps[$. In particular, by combining with the previous result, we have that $\pi \circ \varphi_t((x_0,\xi_0))\notin E$ for $t\in ]-R_0,R_0+\eps[$. We have constructed a geodesic path of length at least $2R_0+\eps/2$ that does not intersect $E$. This implies, in particular, that $T_{GCC}(E)\geq 2R_0+\eps/2$, which is a contradiction.

We now prove the last statement. By definition, $\dist(x, E) > R_0$ is impossible, so we only have to consider $x_0$ so that $\dist(x_0, E) = R_0$. 
Since $M\setminus \overline{E}$ is an open connected set of $M$ it is also arcwise connected. Let $U$ be a connected set of $M\setminus \overline{E}$ containing $x_0$. We prove $U\subset B(x_0, R_0)$. Let $x\in U$. By assumption, there exists $\gamma$ one continuous path so that $\gamma(0)=x_0$, $\gamma(1)=x$ and $\gamma(t)\in U\subset M\setminus \overline{E}$. In particular, $\gamma(t)\notin \overline{E}$ for $t\in [0,1]$. Assume $d(x_0,x)\geq R_0$. By continuity, there exists $t\in [0,1]$ so that $d(x_0,\gamma(t))=R_0$. There is a geodesic miminizing the distance between $\gamma(t)$ and $x_0$. That is, there exists $\xi_0 \in S_{x_0}^* M$ so that $\pi\circ \varphi_{R_0}((x_0,\xi_0))=\gamma(t)$. In particular, by the previous statement, $\gamma(t)\in \overline{E}$. This is a contradiction. So, we have proved $U\subset B(x_0,R_0)$, which gives the result.
\enp

\begin{remark}
Note that we have the two equivalences $T_{GCC}(E)=0 \Longleftrightarrow$ ($E$ satisfies GCC and $\ovl{E}  = M$), and $T_{UC}(E) = 0 \Longleftrightarrow \ovl{E}  = M$.
\end{remark}

The following result is used in Section~\ref{s:low-freq-T}.
\begin{lemma}
\label{l:omega-0}
Let $\omega$ be an open subset of $M$ satisfying GCC and such that $T_{UC}(\omega) < T_{GCC}(\omega)$. Then, there exists an open subset $\omega_0$ of $M$ such that 
$$
\ovl{\omega}_0 \subset \omega ,  \quad \text{and} \quad T_{UC}(\omega_0) < T_{GCC}(\omega) .
$$
\end{lemma}
 \bnp
We prove the more general fact for an open set $\omega\subset M$:
\bnan
\label{e:epseps}
\text{For any $\eps>0$, there exists an open set $\omega_0$ with $\ovl{\omega}_0 \subset \omega$ so that $\mathcal{L}(M ,\omega_0)\leq \mathcal{L}(M ,\omega)+\eps$.}
\enan
By compactness of $M$, we can find a finite sequence of points $(x_i)_{i\in I}$ with $I$ finite, so that 
\bna
M=\cup_{i\in I} B(x_i,\eps/2),
\ena
where $B(y,r) = \{x \in M, \dist(x,y)<r\}$.
By definition of $\mathcal{L}(M ,\omega)$, for any $x_i$, we have $\dist(x_i,\omega)\leq \mathcal{L}(M ,\omega)$ and there exists $y_i\in \omega$ so that $\dist(x_i,y_i)\leq \mathcal{L}(M ,\omega)+\eps/2$.

Now, since $\omega$ is an open set, there exists $r_i$ so that $\overline{B(y_i,r_i)}\subset \omega$. Now, we take 
\bna
\omega_0:=\cup_{i\in I}B(y_i,r_i).
\ena
For any $x\in M$, we can pick $i\in I$ so that $x\in B(x_i,\eps/2)$. In particular, $\dist(x,y_i)\leq \dist(x,x_i)+\dist(x_i,y_i)\leq \mathcal{L}(M ,\omega)+\eps$. Therefore, for any $x\in M$, we have $\dist(x,\omega_0)\leq \mathcal{L}(M ,\omega)+\eps$. This gives $\mathcal{L}(M ,\omega_0)\leq \mathcal{L}(M ,\omega)+\eps$. That $\ovl{\omega}_0 \subset \omega$ comes from $\overline{B(y_i,r_i)}\subset \omega$ for all $i\in I$ and the finiteness of $I$. This concludes the proof of~\eqref{e:epseps}, and hence of the lemma.
\enp

\bibliographystyle{alpha}
\bibliography{bibli}

\end{document}